\documentclass{article}
\usepackage[utf8]{inputenc}
\usepackage{geometry}[margin=1in]
\usepackage{mathtools} 
\usepackage{comment}
\usepackage{amsmath}
\usepackage{amssymb}
\usepackage{amsthm}
\usepackage{enumitem}
\usepackage{xcolor}
\usepackage{svg}
\usepackage{pst-node}
\usepackage{tikz-cd} 
\usepackage[
backend=biber,
style=alphabetic,
]{biblatex}
\usepackage[hidelinks]{hyperref}

\numberwithin{equation}{section}

\newtheorem{theorem}{Theorem}[section]

\newtheorem{proposition}[theorem]{Proposition}
\newtheorem{corollary}[theorem]{Corollary}
\newtheorem{lemma}[theorem]{Lemma}
\newtheorem{lemma-definition}[theorem]{Lemma-Definition}

\newtheorem{conjecture}[theorem]{Conjecture}

\theoremstyle{definition}
\newtheorem{definition}[theorem]{Definition}

\newtheorem{remark}[theorem]{Remark}

\newtheorem{notation}[theorem]{Notation}

\newtheorem*{acknowledgements*}{Acknowledgements}
\newtheorem*{organization*}{Organization}

\newcommand{\R}{\mathbb{R}}
\newcommand{\C}{\mathbb{C}}
\newcommand{\Z}{\mathbb{Z}}

\newcommand{\F}{\mathbb{F}}
\newcommand{\ind}{\operatorname{ind}}
\newcommand{\wind}{\operatorname{wind}}

\DeclarePairedDelimiter{\ceil}{\lceil}{\rceil}
\DeclarePairedDelimiter{\floor}{\lfloor}{\rfloor}

\title{A connected sum formula for embedded contact homology}
\author{Luya Wang}
\date{}

\begin{document}

\maketitle
\begin{abstract}
    Given two closed contact three-manifolds, one can form their contact connected sum via the Weinstein one-handle attachment. We study how pseudo-holomorphic curves in the symplectization behave under this operation. As a result, we give a connected sum formula for embedded contact homology.
\end{abstract}

\section{Introduction}
\subsection{Embedded contact homology}
Embedded contact homology (ECH) is an invariant for closed contact three-manifolds introduced by Hutchings in \cite{Hutchings_index_inequality, Hutchings}. We assume for simplicity that our three-manifolds are connected. Given a contact three-manifold $(Y,\xi)$, fix a contact one-form $\lambda$ such that $\ker(\lambda) = \xi$, i.e., $\lambda \wedge d\lambda >0$, and all Reeb orbits associated to $\lambda$ are nondegenerate. Given $\Gamma\in H_1(Y)$, the ECH chain complex is a free $\F$-module generated by certain
finite sets of Reeb orbits in the homology class $\Gamma$, where $\F:=\Z/2\Z$. The ECH differential counts pseudo-holomorphic curves with asymptotic ends at Reeb orbits and ECH index one in the symplectization of the contact three-manifold, with respect to a certain choice of $\R$-invariant almost complex structure $J$. It is shown in \cite{HT_gluing1,HT_gluing2} that $\partial^2 = 0$ for the ECH differential. Let $ECC_*(Y,\lambda, \Gamma, J)$ denote the ECH chain complex and let $ECH_*(Y,\lambda, \Gamma, J)$ denote the homology of the ECH chain complex. We often suppress the notation $\Gamma$ and understand that ECH decomposes as 
$$ECH_*(Y,\lambda, J) = \oplus_{\Gamma\in H_1(Y)}ECH_*(Y,\lambda, \Gamma, J).$$

ECH is a priori dependent on the choices of contact form $\lambda$ and almost complex structure $J$. In \cite{Taubes_isomorphism_I,Taubes_isomorphism_II, Taubes_isomorphism_III, Taubes_isomorphism_IV, Taubes_isomorphism_V}, Taubes showed an isomorphism between ECH and a certain version of Seiberg-Witten Floer cohomology:
\begin{equation}
\label{eq:Taubes_isom}
    ECH_*(Y,\lambda, \Gamma, J) \cong \widehat{HM}^{-*}(Y, \mathfrak{s}(\xi) + \text{PD}(\Gamma))
\end{equation}
as relatively graded $\F$-modules, where $\mathfrak{s}(\xi)$ is the spin-c structure determined by the $2$-plane field $\xi$. For the definition of Seiberg-Witten Floer homology, see \cite{Kronheimer_Mrowka}. The isomorphism (\ref{eq:Taubes_isom}) establishes the well-definedness of ECH. In particular, ECH is independent of the choices of almost complex structures and contact forms and is sometimes denoted as $ECH_*(Y,\xi, \Gamma)$.

ECH is also equipped with a degree $-2$ chain map called the \emph{$U$ map}. The chain level $U$ map is defined by counting ECH index $2$ curves passing through a generic base point, which depends on the choice of the base point. The induced $U$ map on homology does not depend on such a choice and endows ECH with an $\F[U]$-module structure.

\subsection{The main theorem}
Given two contact three-manifolds $(Y_1, \xi_1)$ and $(Y_2, \xi_2)$, one can form their contact connected sum $(Y_1 \# Y_2, \xi_1 \# \xi_2)$ by the Weinstein one-handle attachment \cite{Weinstein_one_handle}. Up to contactomorphism, this is a well-defined operation.

Define \emph{the derived tensor product} of two chain complexes $C_1$ and $C_2$ over $\F[U]$ to be 
$$C_1 \tilde\otimes_{\F[U]} C_2:= H_*(C_1 \otimes_{\F} C_2 \xrightarrow{U_1 \otimes id + id \otimes U_2} C_1 \otimes_{\F} C_2[-1]),$$
where the right hand side denotes the homology of the \emph{mapping cone} of the chain map $U_1 \otimes id + id \otimes U_2$ from $C_1 \otimes_{\F} C_2$ to $C_1 \otimes_{\F} C_2[-1]$, also sometimes denoted as $Cone(U_1 \otimes id + id \otimes U_2)$ in our paper. For details of the definition see Section \ref{sec:mapping_cone}.

We now state the main theorem:

\begin{theorem}
\label{thm:main}
Let $(Y_1, \lambda_1)$ and $(Y_2, \lambda_2)$ be two closed connected contact three-manifolds with given nondegenerate contact forms $\lambda_1$ and $\lambda_2$. Then,
\begin{equation}
\label{eq:derived_tensor_product}
    ECH(Y_1 \# Y_2, \xi_1\#\xi_2, \Gamma_1+\Gamma_2) \cong ECH(Y_1,\xi_1, \Gamma_1) \tilde\otimes_{\F[U]} ECH(Y_2,\xi_2, \Gamma_2)
\end{equation}
as $\F$-modules.
\end{theorem}
\begin{remark}
    Similar theorems to our main theorem are known for Seiberg-Witten Floer homology \cite{BMO, HM_HF_5, Lin_connected_sum} and Heegaard Floer homology \cite{OSz, OSz_connected_sum_formula}. See \cite{Taubes_isomorphism_I,Taubes_isomorphism_II, Taubes_isomorphism_III, Taubes_isomorphism_IV, Taubes_isomorphism_V, HF_ECH, ECH_HF_1, ECH_HF_2, ECH_HF_3} for the isomorphisms between ECH and Seiberg-Witten Floer and between ECH and Heegaard Floer. See also \cite{HM_HF_1, HM_HF_2, HM_HF_3, HM_HF_4, HM_HF_5}. Our motivations to have a connected sum formula proven in ECH include potential generalizations to other contact homologies and obtaining more information on quantitative invariants in ECH. More details will be discussed in Section \ref{sec:future}.
\end{remark}

Theorem \ref{thm:main} can also be used to compute ECH of subcritical surgery. Note that the subcritical surgery that is attaching a $1$-handle to the three-manifold $Y$ itself is simply a self connected sum. Topologically, this is the same as performing a connected sum with $S^1 \times S^2$. The ECH of $S^1 \times S^2$ has been computed in Section 12.2.1 in \cite{T3}. One can also understand the $U$ maps on $ECH(S^1 \times S^2)$ from Seiberg-Witten theory \cite[Section 36]{Kronheimer_Mrowka}. Given a contact structure $\xi_0$ on $S^1 \times S^2$, let $\Gamma_0\in H_1(S^1\times S^2)$ be the homology class such that $c_1(\xi)+2PD(\Gamma_0) = 0$. Then, 
$$ECH(S^1 \times S^2, \xi_{0}, \Gamma) = 
\begin{cases}
(\F[U, U^{-1}]/U\F[U]) \otimes H_*(S^1;\Z), \Gamma = \Gamma_0\\
0, \text{ otherwise}
\end{cases}$$ 
as $\F[U]$-modules. Now by using (\ref{eq:derived_tensor_product}) and K\"{u}nneth formula, we have:
\begin{corollary}
Given a closed connected contact three-manifold $(Y, \xi)$,
$$ECH(Y \# (S^1 \times S^2), \xi \# \xi_0, \Gamma+\Gamma_0) \cong ECH(Y, \xi, \Gamma) \otimes_{\F} H_*(S^1;\Z)$$
as $\F$-modules.
\end{corollary}

\subsection{Ideas of the proof of the main theorem}
To understand ECH chain complex of the contact connected sum $(Y_1 \# Y_2, \xi_1 \# \xi_2)$, one needs to understand not only the contact structure but also the contact form on the connected sum. Let $\lambda_i$ be a nondegenerate contact form such that $\ker{\lambda_i} = \xi_i$ for $i=1,2$. Given two Darboux charts in $Y_1$ and $Y_2$, a particular model of the contact connected sum $(Y_1 \# Y_2, \lambda_1 \# \lambda_2)$ is carefully described in \cite{Fish_Siefring} and depends on various choices. In particular, the connected sum sphere $S_+$ which is contained in the ascending manifold of the Weinstein one-handle contains a hyperbolic Reeb orbit which is the equator of $S_+$. This hyperbolic orbit is denoted as the \emph{special hyperbolic orbit} $h$ throughout our paper. In addition, one may adjust the size of the Weinstein one-handle in order to control the radius of $S_+$. Let $(Y_1 \#_{R} Y_2, \lambda_1 \#_R \lambda_2)$ denote the resulting contact connected sum with a connected sum sphere of radius $R$. More details about the contact connected sum operation will be given in Section \ref{sec:connected_sum}. 

The proof of Theorem \ref{thm:main} goes through a chain level statement on the filtered ECH chain complex. Recall that there is a filtration on ECH chain complex by the symplectic action functional integrating the contact form over Reeb orbits. The filtered ECH chain complex $ECC^L(Y,\lambda, \Gamma, J)$ is generated by orbit sets up to symplectic action $L$. To ease notations, we suppress $\Gamma$ and $J$ from input of ECH chain complex when it is clear. 

Let
$$C_o:= ECC(Y_1, \lambda_1) \otimes_{\F} ECC(Y_2, \lambda_2)$$ 
denote all orbit sets in $Y_1 \sqcup Y_2$ counted as ECH generators. Now note that up to action $L$, we may ignore orbits that cross the connected sum region by shrinking the connected sum sphere to be of a small enough radius $R(L)$, since these orbits would have actions greater or equal to $L$ by a compactness argument as in Lemma \ref{lem:large_energy}. Therefore, there is an obvious identification on the vector space level between the filtered ECH complex $ECC^L(Y_1 \#_{R(L)} Y_2, \lambda_1\#_{R(L)}\lambda_2)$ and the filtered mapping cone complex 
$$Cone^L(U_1 \otimes id + id \otimes U_2):=C_o^L \oplus C_h^L,$$
where $C_h \cong C_o$ by appending the special hyperbolic orbit $h$ and the superscript $L$ denotes the filtration. In fact, one may find a chain homotopy equivalence that is triangular with respect to this obvious identification of the vector spaces:

\begin{proposition}
\label{prop:main_filtered_intro}
Given two closed connected contact three-manifolds $(Y_1, \lambda_1)$ and $(Y_2, \lambda_2)$ with nondegenerate contact forms $\lambda_i$, there exists a strictly decreasing function $R: \R \to \R$ with $$\lim_{L\to \infty} R(L) = 0,$$
such that there is a chain homotopy equivalence 
$$f: ECC^L(Y_1 \#_{R(L)} Y_2, \lambda_1\#_{R(L)}\lambda_2) \longrightarrow Cone^L(U_1 \otimes id + id \otimes U_2).$$
\end{proposition}

More details about the identification on the vector space level will be provided in Section \ref{sec:correspondence}. The chain homotopy equivalence in Proposition \ref{prop:main_filtered_intro} will be constructed in Section \ref{sec:chain_homotopy}.

\subsection{Further directions}
\label{sec:future}
From the perspective of Symplectic Field Theory (SFT) constructed by Eliashberg, Givental and Hofer \cite{EGH}, one could also try to prove a connected sum formula in linearized contact homology. This has been studied by Bourgeois and van-Koert \cite{BVK} and Yau \cite{Yau}. We hope that our connected sum formula for embedded contact homology could give ideas to such formulas for other contact homologies that involve higher genus curves and more asymptotic ends.

It would also be interesting to study the ECH formulas and cobordism maps under additional contact surgeries. One natural candidate that is suitable for the techniques of the present paper is to study how ECH behaves when cut along mixed tori studied by \cite{Menke_JSJ_mixed_torus}.  

Theorem \ref{thm:main} computes ECH of a connected sum as an $\F$-module. However, the model of the derived tensor product on the right hand side of (\ref{eq:derived_tensor_product}) has a natural $U$-action, $U_1 \otimes id$, which is homotopic to $id \otimes U_2$. It is natural to ask whether Theorem \ref{thm:main} holds over $\F[U]$-modules. This amounts to studying ECH $U$ maps on the chain level on the connected sum, which is also important for studying more refined invariants in ECH, such as the ECH spectral invariants. This will be discussed in future works, but we present here a conjecture.

Suppose $\lambda$ is nondegenerate. Let $0\neq \sigma\in ECH(Y,\lambda, \Gamma)$. Define $c_\sigma(Y,\lambda)$ to be the infimum over $L\in \R$ such that $\sigma$ is in the image of the inclusion-induced map 
$$ECH^L(Y,\lambda, \Gamma) \longrightarrow ECH(Y,\lambda, \Gamma).$$ 
Recall that there is a canonical element $c(\xi):=[\emptyset]\in ECH(Y, \lambda, 0)$ called the \emph{ECH contact invariant}.

\begin{definition}
If $(Y,\lambda)$ is a closed connected contact three-manifold with the contact invariant $c(\xi) \neq 0$ and if $k$ is a nonnegative integer, then define the \emph{$k$-th ECH spectral invariant} to be
$$c_k(Y,\lambda):= \inf\{c_\sigma(Y,\lambda)|\sigma\in ECH(Y, \lambda, 0), U^k\sigma = [\emptyset]\}.$$
\end{definition}

\begin{conjecture}
\label{conj:ECH_spectrum_connected_sum}
$$\lim_{R\to 0}c_k((Y_1\#_{R}Y_2, \lambda_1 \#_{R} \lambda_2)) = \max\{c_i(Y_1, \lambda_1)+c_j(Y_2, \lambda_2)|i+j = k\}.$$
\end{conjecture}
The conjecture holds for the boundaries of irrational ellipsoids by applying the explicit formulas of differentials in the connected sums as in Section \ref{sec:differentials}, the ``no-crossing'' Lemma \ref{lem:no_crossing}, as well as previously known results on the U maps and differentials in the boundaries of irrational ellipsoids \cite{quantitative_ECH}.

\begin{organization*} Section \ref{sec:overview} reviews basic definitions of embedded contact homology. Section \ref{sec:connected_sum} discusses the Reeb dynamics of the contact connected sum and the asymptotic behaviors of pseudo-holomorphic curves in its symplectization. Section \ref{sec:correspondence} discusses how to ignore potential Reeb orbits that cross the connected sum region and identifies a filtered ECH chain complex of a connected sum with a filtered mapping cone complex associated to the ECH $U$ maps on the level of vector spaces. Section \ref{sec:differentials} relates some of the new ECH differentials in the connected sum to the ECH differentials in the original contact three-manifolds. Section \ref{sec:chain_homotopy} relates the remaining differentials to the ECH $U$ maps in the original contact three-manifolds and constructs a chain homotopy equivalence that proves Proposition \ref{prop:main_filtered_intro}. Section \ref{sec:direct_limit} uses a direct limit argument similar to that in \cite{Nelson_Morgan} to prove the main result Theorem \ref{thm:main}. \end{organization*}

\begin{acknowledgements*} The author is deeply indebted to her advisor Michael Hutchings for constant support and enlightening conversations. We would also like to thank Otto van Koert for helpful communications and sharing their draft on the linearized contact homology side. We thank Ian Zemke for discussions of homological algebra and Heegaard Floer theory. We thank Oliver Edtmair, John Etnyre, Francesco Lin, Juan Mu\~{n}oz-Ech\'{a}niz, Hiro Lee Tanaka, Masaki Taniguchi, Morgan Weiler, Yuan Yao and Ziwen Zhao for helpful conversations. In addition, we thank Morgan Weiler and the anonymous referee for their comments and suggestions in an earlier draft. The author acknowledges support by the NSF GRFP under Grant DGE 2146752. \end{acknowledgements*}

\section{An overview of ECH}
\label{sec:overview}
In this section we give a very quick overview of ECH. See \cite{Hutchings_index_inequality,Hutchings_index_revisited,Hutchings} for more details. Let $Y$ be a closed connected three-manifold with a contact form $\lambda$. Let $R$ be the \emph{Reeb vector field} determined by $d\lambda(R, \cdot) = 0$ and $\lambda(R) = 1$. Over each closed Reeb orbit $\gamma: \mathbb{R}/T\mathbb{Z} \to Y$, where $T$ is the period of the orbit, the linearized Reeb flow determines a symplectic linear map $$d\psi_T:T_{\gamma(0)} Y \longrightarrow T_{\gamma(T)}Y.$$ 
If $d\psi_T$ does not have $1$ as an eigenvalue, we say that $\gamma$ is \emph{nondegenerate}. A contact form $\lambda$ is said to be \emph{nondegenerate} if all Reeb orbits are nondegenerate. Note that this is a generic condition and from now on we fix $\lambda$ to be nondegenerate.

Let $\mathcal{P}$ be the set of embedded Reeb orbits of the Reeb vector field associated to $\lambda$. The ECH chain complex $ECC(Y,\lambda, J)$ is generated as a $\F$ vector space by \emph{orbit sets} $\gamma = \{(\gamma_i, m_i)\}$, such that:
\begin{itemize}
\item $\gamma_i \in \mathcal{P}$ are distinct Reeb orbits;
\item $m_i\in \Z^+$ is the covering multiplicity of the orbit $\gamma_i$;
\item if $\gamma_i$ is hyperbolic, i.e. the linearized return map has real eigenvalues, then $m_i = 1$.
\end{itemize}

We call orbit sets satisfying the above criteria \emph{admissible}.

Let $(\R_s \times Y, d(e^s\lambda))$ be the \emph{symplectization} of $(Y, \lambda)$. Let $J$ be a \emph{$\lambda$-adapted} almost complex structure. This means that $J$ is $\mathbb{R}$-invariant, $J(\partial_s) = R$ where $R$ is the Reeb vector field associated to $\lambda$, and $J$ sends the contact structure $\xi:=\ker(\lambda)$ to itself, rotating $\xi$ positively with respect to the orientation on $\xi$ given by $d\lambda$. We consider $J$-holomorphic curves $u:(\Sigma, j) \to (\R \times Y, J),$ where $(\Sigma, j)$ is a punctured Riemann surface, up to biholomorphisms. A \emph{positive asymptotic end} of $u$ at a Reeb orbit $\gamma$ is a puncture with a neighborhood that can be given coordinates of a \emph{positive half-cylinder} $(s, t)\in [0,\infty) \times (\R/T\Z)$ such that $j(\partial s) = t$, $\lim_{s\to \infty} \pi_\R(u(s, t)) = \infty$ and $\lim_{s \to \infty} \pi_Y(u(s, \cdot)) = \gamma$. A \emph{negative asymptotic end} is defined analogously, where the neighborhood of a negative puncture is identified with a \emph{negative half-cylinder} $(-\infty, 0] \times (\R/T\Z)$.

The ECH differential counts \emph{$J$-holomorphic currents} of \emph{ECH index} one. A \emph{$J$-holomorphic current} is a finite set of pairs $\mathcal{C} = \{(C_k, d_k)\}$, where $C_k$ are distinct, connected, somewhere injective $J$-holomorphic curves in $(\R \times Y, d(e^s\lambda))$ and $d_k\in \Z^+$. We say that $\mathcal{C}$ is \emph{positively asymptotic to} an orbit set $\alpha = \{(\alpha_i, m_i)\}$, if $C_k$ has positive ends at covers of $\alpha_i$ with multiplicity $cov(C_k)$ and $\sum_k cov(C_k) = m_i$. \emph{Negative asymptotic} as currents is defined analogously. We now review ingredients of the ECH index and related properties.

\subsection{The Conley-Zehnder index}
\label{sec:CZ}
In this subsection we discuss a classical quantity that measures the ``winding'' of the linearized return map of a given Reeb orbit, with respect to a given trivialization.

\begin{definition}
    Let $\tau$ be a trivialization of the contact structure $\xi$ over the Reeb orbit $\gamma$, parameterized as $\gamma:\mathbb{R}/T\mathbb{Z} \to Y$. The linearization of the Reeb flow $d\psi_t:T_{\gamma(0)} Y \to T_{\gamma(t)}Y$ induces a family of $2\times 2$ symplectic matrices $\phi_t:\xi_{\gamma(0)} \to \xi_{\gamma(t)}$ with respect to $\tau$. The \emph{Conley-Zehnder index over $\gamma$} denoted as $CZ_\tau(\gamma)\in \mathbb{Z}$ is the Conley-Zehnder index of the family of symplectic matrices $\phi_{t\in [0,T]}$.
\end{definition}

Since we assumed the contact form we started with is \emph{nondegenerate}, the linearized return map $\phi_T(\gamma)$ does not have $1$ as an eigenvalue\footnote{Indeed, we may not have any roots of unity as eigenvalues of the linearized return map since we may consider covers of Reeb orbits.}. If $\phi_T(\gamma)$ has eigenvalues in the unit circle, we call $\gamma$ \emph{elliptic}. Otherwise, $\gamma$ is \emph{hyperbolic}. If $\gamma$ is hyperbolic, we may choose $v\in \mathbb{R}^2$ to be an eigenvector of $\phi_T$ and the family $\{\phi_t(v)\}_{t\in [0,T]}$ rotates by $k\pi$, where $k\in \mathbb{Z}$. In this case, 
$$CZ_\tau(\gamma) = k.$$ 
If $\gamma$ is elliptic, we may adjust the trivialization $\tau$ so that $\phi_t$ is a rotation by angle $2\pi \theta_t$, where $\theta_t$ is continuous in $t\in [0,T]$ and $\theta_0 =0$. In this case, $$CZ_\tau(\gamma) = 2\floor{\theta}+1,$$
where $\theta= \theta_T$.

\subsection{The relative first Chern class}
\label{sec:c1}
The relative first Chern class is a generalization of the usual first Chern class of a complex line bundle $\xi$ over a curve with boundaries. 
\begin{definition}
    Let $\alpha = \{(\alpha_i, m_i)\}$ and $\beta = \{(\beta_j, n_j)\}$ be orbit sets with $[\alpha] = [\beta]\in H_1(Y)$, where $[\alpha] = \sum_i m_i [\alpha_i]$ and $[\beta] = \sum_j n_j [\beta_j]$. Then $H_2(Y, \alpha, \beta)$ denotes the set of relative homology classes of $2$-chains $Z$ in $Y$ such that
    $$\partial Z = \sum_i m_i\alpha_i - \sum_j n_j \beta_j.$$
\end{definition}

\begin{definition}
Fix a trivialization $\tau$ of the contact structure $\xi$ over the Reeb orbits. Let $Z\in H_2(Y, \alpha, \beta)$ where $\alpha$ and $\beta$ are orbit sets, and $f:S \to Y$ be a smooth representative of $Z$. The \emph{relative first Chern class}
$$c_{\tau}(Z):= c_1(\xi|_Z, \tau) \in \Z$$ 
counts algebraically the zeros of a generic section $\psi$ of the bundle $f^*\xi \to S$ transverse to the zero section, which is non-vanishing and has zero winding number with respect to $\tau$ on the boundary.
\end{definition}

The relative first Chern class $c_\tau(Z)$ is well-defined and changes under the relative second homology class by
\begin{equation}
    c_\tau(Z) - c_\tau(Z') = \langle c_1(\xi), Z-Z'\rangle
\end{equation}
where $c_1(\xi) \in H^2(Y; \Z)$ is the ordinary first Chern class. Let $\tau':= (\{{\tau_i^+}'\}, \{{\tau_j^-}'\})$ be another trivialization over the positive orbit set $\alpha:= \{(a_i,m_i)\}$ and negative orbit set $\beta:=(\{b_j,n_j\})$, then
\begin{equation}
    c_\tau(Z) - c_{\tau'}(Z) = \sum_i m_i(\tau_i^+ - {\tau_i^{+}}') - \sum_j n_j(\tau_j^- - {\tau_j^-}'),
\end{equation}
where $\tau_i^\bullet - {\tau_i^{\bullet}}'$ denotes the change in the homotopy classes of trivializations in $\mathbb{Z}$.

\subsection{The relative intersection pairing}
The relative intersection pairing associates an intersection number to a pair of relative second homology classes. It is closely related to the behavior of the \emph{asymptotic linking number} which will be defined in Section \ref{sec:asymptotic_h}.

\begin{definition}
Let $\alpha = \{(\alpha_i, m_i)\}$ and $\beta = \{(\beta_j, n_j)\}$ be orbit sets with $[\alpha] = [\beta] \in H_1(Y)$ and let $Z\in H_2(Y, \alpha, \beta)$. An \emph{admissible representative} of $Z$ is a smooth map $f: S \to [-1,1] \times Y$, where $S$ is a compact oriented surface with boundary, such that:
\begin{itemize}
\item $f|_{\partial S}$ consists of positively oriented covers of $\{1\} \times \alpha_i$ with total multiplicity $m_i$ and negatively oriented covers of $\{-1\} \times \beta_j$ with total multiplicity $n_j$;
\item the composition of $f$ with the projection $[-1,1] \times Y \to Y$ represents the class $Z$; 
\item $f|_{\dot{S}}$ is an embedding and $f$ is transverse to $\{-1,1\} \times Y$. 
\end{itemize}
\end{definition}

Now fix a trivialization $\tau$ of the contact structures $\xi$ over all Reeb orbits.
\begin{definition}
\label{def:relative_intersection}
Let $Z\in H_2(Y, \alpha, \beta)$ and $Z'\in H_2(Y, \alpha', \beta')$. The \emph{relative intersection pairing} $Q_\tau(Z, Z'): H_2(Y, \alpha, \beta) \times H_2(Y, \alpha', \beta') \to \mathbb{Z}$ is defined by
$$Q_\tau(Z,Z') := \#(\dot{S} \cap \dot{S}') - l_\tau(S,S'),$$
where $S$ and $S'$ are admissible representatives of $Z$ and $Z'$, and the interiors $\dot{S}$ and $\dot{S'}$ are transverse and do not intersect near the boundary. The term $l_\tau(S,S')$ is the \emph{asymptotic linking number}, which will be defined in Section \ref{sec:asymptotic_h}.
\end{definition}

An alternative definition for $Q_\tau(Z,Z')$ is to find representatives $S$ of $Z$ and $S'$ of $Z'$, called \emph{$\tau$-representatives}, such that the projection of their intersections under the restricted projection maps $(1-\epsilon, 1] \times Y \to Y$ and $[-1, -1+\epsilon) \times Y \to Y$ are embeddings, and that near each orbit $\gamma_i$ at their ends, $S$ and $S'$ are unions of rays emanating from points of $\gamma_i$, and the rays do not intersect and do not rotate with respect to the trivialization $\tau|_{\gamma_i}$. One can check that these two definitions are the same (Lemma 8.5 in \cite{Hutchings_index_inequality}). In fact, both definitions will be useful to us. 

Both definitions of $Q_\tau(Z,Z')$ are well-defined by Lemma 2.5 and 8.5 in \cite{Hutchings_index_inequality}. In particular, $Q_\tau$ is independent of the choice of admissible representatives of the relative second homology classes. Moreover, under change of second relative homology class,
\begin{equation}
\label{eq:change_of_homology_class_for_Q_tau}
    Q_\tau(Z_1,Z') - Q_\tau(Z_2, Z') = (Z_1 - Z_2) \cdot [\alpha']
\end{equation}
and under change of trivilizations, 
\begin{equation}
    Q_\tau(Z, Z') - Q_{\tau'}(Z,Z') = \sum_i m_i {m_i}' ({\tau^+_i} - {\tau^+_i}') - \sum_j n_j {n_j}' ({\tau^-_j} - {\tau^-_j}').
\end{equation}

We call $Q_\tau(Z):=Q_\tau(Z,Z)$ the \emph{relative self-intersection pairing}.

\begin{notation}
In the ECH literature, we often abuse notation and let $C$ denote the holomorphic curve $u: (C,j) \to (X,J)$. In addition, we often write $c_\tau(C):=c_\tau([C])$ and $Q_\tau(C):=Q_\tau([C])$.
\end{notation}

\subsection{The Fredholm index and the ECH index}

Let $\mathcal{M}(\gamma_1^+, \cdots, \gamma_k^+; \gamma_1^-, \cdots, \gamma_l^-; J)$ denote the moduli space of $J$-holomorphic curves with positive ends at $\gamma_1^+, \cdots, \gamma_k^+$ and negative ends at $\gamma_1^-, \cdots, \gamma_l^-$. 
\begin{definition}
    The \emph{Fredholm index} of a curve $C\in \mathcal{M}({\gamma_1}^+, \cdots, {\gamma_k}^+; {\gamma_1}^-, \cdots, {\gamma_l}^-; J)$ is defined by 
    $$\ind(C) := -\chi(C) + 2c_{\tau}(C) + \sum_{i=1}^k CZ_{\tau}(\gamma_{i}^+) - \sum_{j=1}^l CZ_{\tau}(\gamma_{j}^-).$$
\end{definition}

If $J$ is generic and $C$ is simple, then the above moduli space is a manifold near $C$ of dimension $ind(C)$ by \cite{Dragnev}. Now, let $\mathcal{M}(\alpha,\beta;J)$ denote the set of $J$-holomorphic currents that approach $\alpha = \{(\alpha_i, m_i)\}$ as a current for $s\rightarrow +\infty$ and $\beta = \{(\beta_j, n_j)\}$ as a current for $s\rightarrow -\infty$.

\begin{definition}
\label{def:ECHindex}
Let $Z\in H_2(Y, \alpha, \beta)$, then the \emph{ECH index} is
$$I(\alpha, \beta, Z) := c_{\tau}(Z) + Q_{\tau}(Z) + CZ^I_{\tau}(\alpha, \beta),$$
where
$$CZ^I_{\tau}(\alpha, \beta) := \sum_i\sum_{k=1}^{m_i} CZ_{\tau}(\alpha_i^k) - \sum_j \sum_{k=1}^{n_j} CZ_{\tau}(\beta_j^k).$$
Let $\mathcal{C}\in \mathcal{M}(\alpha, \beta; J)$. Then we define $I(\mathcal{C}) := I(\alpha, \beta, [\mathcal{C}])$.
\end{definition}

ECH index has many useful properties. The following inequality relates it to the Fredholm index.

\begin{theorem}[Index inequality]\cite[Theorem 4.15]{Hutchings_index_revisited}
\label{thm:index_inequality}
Let $\alpha$ and $\beta$ be orbit sets and $C\in \mathcal{M}(\alpha, \beta; J)$ be a somewhere injective curve, then
$$ind(C)\leq I(C)$$
with equality if and only if $C$ is embedded.
\end{theorem}

The inequality in Theorem \ref{thm:index_inequality} has in particular the following applications.

\begin{theorem}[Low-index currents]\cite[Proposition 3.7]{Hutchings}
\label{thm:low_index}
Suppose $J$ is generic. Let $\alpha$ and $\beta$ be orbit sets and $\mathcal{C} = \{(C_i, d_i)\}\in \mathcal{M}(\alpha, \beta;J)$, be a not necessarily somewhere injective $J$-holomorphic current. Then:
\begin{itemize}
    \item $I(\mathcal{C}) \geq 0$, with equality if and only if each $C_i$ is a trivial cylinder with multiplicity;
    \item If $I(\mathcal{C})=1$, then $\mathcal{C} = C_0 \sqcup C_1$, where $I(C_0) = 0$ and $C_1$ is an embedded curve with $I(C_1) = ind(C_1) = 1$;
    \item If $I(C)=2$, and $\alpha$ and  $\beta$ are admissible orbit sets, then $\mathcal{C} = C_0 \sqcup C_2$, where $I(C_0) = 0$ and $C_2$ is an embedded curve with $I(C_2) = ind(C_2) = 2$.
\end{itemize}
\end{theorem}

In particular, whenever we are counting low-index currents, we may assume that the underlying curves, other than potentially branched covers of trivial cylinders, are embedded. Another useful result relating the different quantities introduced in the previous subsections is the relative adjunction formula, which often helps with ECH index computations.

\begin{proposition}[Relative adjunction formula]\cite[Proposition 4.9]{Hutchings_index_revisited}
\label{prop:adjunction}
    Let $C\in \mathcal{M}(\alpha;\beta;J)$ be somewhere injective. Then $C$ has finitely many singularities and
    \begin{equation}
        c_\tau (C) = \chi(C) + Q_\tau(C) + w_\tau(C) - 2\delta(C),
    \end{equation}
    where $\delta(C)$ is an algebraic count of singularities and $w_\tau(C)$ is the asymptotic writhe to be discussed in Section \ref{sec:asymptotic_h}.
\end{proposition}

\subsection{The ECH differential and the \emph{U} map}

The ECH differential coefficient $\langle \partial \alpha, \beta\rangle$ is defined to be the mod $2$ count of the ECH index $1$ pseudo-holomorphic currents in $\mathcal{M}(\alpha,\beta;J)$. It is shown in \cite{HT_gluing1,HT_gluing2} that $\partial^2 = 0$. As given by the isomorphism (\ref{eq:Taubes_isom}), $ECH(Y, \lambda, J)$ of this chain complex does not depend on the choice of $J$ or $\lambda$.

In addition, ECH is equipped with a degree $-2$ chain map given a base point $z\in Y$ not on a Reeb orbit. Given admissible orbit sets $\alpha$ and $\beta$, we define
$$U_z: ECC_*(Y, \lambda, J) \longrightarrow ECC_{*-2}(Y, \lambda, J),$$
where $\langle U\alpha, \beta\rangle$ is the mod $2$ count of pseudo-holomorphic currents of ECH index $I=2$ from $\alpha$ to $\beta$ passing through $(0,z)\in \R \times Y$. Furthermore, one can show that $U_z$ is a chain map and that up to chain homotopy, $U_z$ does not depend on the choice of the base point $z$. More precisely, let $z'\in Y$ be a different choice of the base point, then it is shown in \cite{HT_weinstein} that there is a chain homotopy between $U_z$ and $U_{z'}$. We will describe this further in Section \ref{sec:partial_ho}.

\subsection{The filtered ECH}
\label{sec:filtered_ECH}
We can associate to a Reeb orbit $\gamma$ a \emph{symplectic action} 
$$\mathcal{A}(\gamma) := \int_\gamma \lambda,$$
and to each orbit set $\alpha = \{(\alpha_i, m_i)\}$ the action
$$\mathcal{A}(\alpha) := \sum_i m_i \mathcal{A}(\alpha_i).$$ 
By Stokes' theorem and the definition of a symplectization-adapted almost complex structure $J$, we know that ECH differential decreases the symplectic action. Therefore, one may define the \emph{filtered} version $ECC^L(Y,\lambda, J)$ to be the subcomplex generated by orbit sets $\alpha$ such that $\mathcal{A}(\alpha)<L$. For $L<L'$, under the inclusion-induced map
$$i_J^{L,L'}:ECH_*^L(Y, \lambda, J) \longrightarrow ECH_*^{L'}(Y,\lambda, J),$$
we may recover
\begin{equation}
\label{eq:direct_lim_J}
    ECH_*(Y,\lambda, J) = \lim_{L\to \infty} ECH_*^L(Y,\lambda, J).
\end{equation}

In the following theorem, we review definitions from \cite{HT_Chord_II} and summarize their results regarding the filtered ECH that we need.

\begin{definition}
Given an action $L$, a contact form $\lambda$ is \emph{L-nondegenerate} if all orbits of action less than $L$ are nondegenerate, and there are no orbit sets of action exactly $L$.
\end{definition}

\begin{definition}
Given an action $L$, an almost complex structure is \emph{$ECH^L$-generic} if the ECH differential is well-defined on admissible orbit sets with actions less than $L$ and satisfies $\partial^2 = 0$.
\end{definition}

\begin{theorem}\cite[Theorem 1.3]{HT_Chord_II}
\label{thm:filtered_ech}
    Let $Y$ be a closed connected oriented three-manifold, then: 
    \begin{enumerate}
        \item If $\lambda$ is an $L$-nondegenerate contact form on $Y$, then $ECH_*^L(Y,\lambda, J)$ does not depend on the choice of $ECH^L$-generic $J$;
        \item If $L<L'$ and $\lambda$ is $L'$-nondegenerate, then $i_J^{L,L'}$ induces a well-defined map
        $$i^{L,L'}:ECH_*^L(Y, \lambda) \longrightarrow ECH_*^{L'}(Y,\lambda);$$
        \item If $\lambda$ is a nondegenerate contact form on $Y$, then $ECH_*(Y,\lambda, J)$ does not depend on the choice of $ECH$-generic $J$, so it may be denoted by $ECH_*(Y,\lambda)$.
    \end{enumerate}
\end{theorem}

\section{The connected sum for contact three-manifolds}
\label{sec:connected_sum}
Let $(Y_1, \lambda_1)$ and $(Y_2, \lambda_2)$ be two closed connected contact three-manifolds with specified contact forms $\lambda_i$. In this paper, we consider their connected sum via the $4$-dimensional Weinstein $1$-handle attachment as in \cite{Weinstein_one_handle}. We follow \cite{Weinstein_one_handle} and \cite{Fish_Siefring} for the explicit descriptions of the Reeb dynamics on the Weinstein $1$-handle model and the pseudo-holomorphic curves in the symplectization. See also \cite{232, BVK}. We give a version here for completeness and to put into our setting.

\subsection{Weinstein 1-handle model and the Reeb flow}
\label{sec:weinstein}

Consider
$$\C^2 = \R^4 = \{(x,y,z,w)\}$$ 
and the standard symplectic form
$$\omega:=dx\wedge dy + dz\wedge dw$$
on it. Consider the Liouville vector field 
$$X = \frac{1}{2}x\partial x + \frac{1}{2} y\partial y + 2z \partial z - w\partial w.$$ 
Now consider the function 
$$f(x,y,z,w) = \frac{1}{4}x^2 +\frac{1}{4}y^2 +z^2 - \frac{1}{2}w^2.$$ 
The hypersurface $\{f=1\}$ is of contact-type. This is because
$$df(X) = \frac{1}{4} x^2 +\frac{1}{4}y^2+4 z^2+w^2>0,$$
so $X$ is transverse to $\{f=1\}$. This is the contact manifold that contains the \emph{ascending sphere} $$S_+ := \{w = 0\} \cap \{f=1\}$$ 
as in \cite{Weinstein_one_handle}. We call $S_+$ the \emph{connected sum sphere}. Similarly, we have the contact-type hypersurface $\{f = -1\}$ and the \emph{descending sphere} $$S_- = \{x=y=z=0\}\cap \{f = -1\}.$$ 
We will see in Lemma \ref{lem:one_handle} that our resulting connected sum $Y_1\# Y_2$ contains the ascending sphere $S_+$. Therefore, in order to understand the Reeb dynamics of $Y_1\# Y_2$, we focus on the contact-type hypersurface $\{f = 1\}$.

We may compute that the contact form 
$$\alpha = i_X \omega|_{f=1} = \frac{1}{2}xdy - \frac{1}{2}ydx + 2zdw + wdz|_{f=1}$$ 
and that $d\alpha = \omega$. We also compute the Reeb vector field 
$$R =\kappa (\frac{1}{2}x\partial y - \frac{1}{2}y \partial x + 2z\partial w + w\partial z)$$ 
where $\kappa:= 1/(1+3z^2+\frac{3}{2} w^2)$. Therefore, one can see that the equator 
$$h:=\{x^2+y^2 = 4, z= w =0\}$$
of $S_+$ is the only orbit in the ascending manifold of the $1$-handle. To compute the Reeb flow at the orbit $h$, we need to solve the linear ODE $$\dot{\eta(t)} = A \eta(t),$$ 
where
\begin{equation*}
A = 
\begin{pmatrix}
0 & \frac{1}{2} & 0 & 0 \\
-\frac{1}{2} & 0 & 0 & 0 \\
0 & 0 & 0 & 1 \\
0 & 0 & 2 & 0
\end{pmatrix}
\end{equation*}
and $\eta(t): = \begin{pmatrix}
x(t) \\ y(t) \\ z(t) \\ w(t)
\end{pmatrix}$. We have the solution $\eta(t) = e^{At} \eta(0)$, where 
\begin{equation}
\label{eq:flow}
\Phi(t): = e^{At} = 
\begin{pmatrix}
-cos(t/2) & sin(t/2) & 0 & 0 \\
sin(t/2) & cos(t/2) & 0 & 0 \\
0 & 0 & cosh(\sqrt{2}t) & sinh(\sqrt{2}t) \\
0 & 0 & \sqrt{2}sinh(\sqrt{2}t) & \sqrt{2}cosh(\sqrt{2}t) 
\end{pmatrix}
\end{equation}

Note that the contact structure restricted to $h$ is $$\xi|_{h} = \langle \partial z, \partial w\rangle.$$
Now consider the trivialization $\tau_0$ of the contact structure $\xi$ over $h$ given by the second factor of $\R^4 = \R^2 \oplus \R^2$. Under this trivialization,
$$CZ_{\tau_0}(h) = 0.$$
This is because the flow does not rotate the contact plane $\langle \partial z, \partial w\rangle$ by looking at the lower $2\times 2$ matrix of $\Phi(t)$.

\begin{definition}
    An \emph{isotropic setup} is a quintuple $(V, \omega, x, Y, \Lambda)$ where $(V, \omega)$ is a symplectic manifold with Liouville vector field $x$ and $\omega = d\lambda$; $Y \subset V$ is a codimension-$1$ hypersurface transverse to $x$; and $\Lambda \subset Y$ is a closed isotropic submanifold for the contact structure $\ker(\lambda|_{Y})$.
\end{definition}

Once we have the model for the connected sum ``tube'', the following proposition by Weinstein tells us how to glue two isotropic setups together via the tube that is the Weinstein $1$-handle in our case. See \cite{Weinstein_one_handle, Cieliebak_Eliashberg} for more details.

\begin{proposition}\cite[Proposition 4.2]{Weinstein_one_handle}
\label{prop:isotropic_setup}
    Let $(V_i, \omega_i, x_i, Y_i, \Lambda_i)$, $i\in \{1,2\}$ be two isotropic setups. Given a diffeomorphism $f:\Lambda_1 \to \Lambda_2$ covered by an isomorphism $\Phi$ of their symplectic subnormal bundles $(T\Lambda_i)^\omega/T\Lambda_i\subset \xi_i$, there exists an isomorphism of isotropic setups
    $$F: (W_1, \omega_1, x_1, Y_1, \Lambda_1) \longrightarrow (W_2, \omega_2, x_2, Y_2, \Lambda_2)$$
    between neighborhoods $W_i$ of $\Lambda_i$ in $V_i$ inducing $f$ and $\Phi$.
\end{proposition}

Then, one may glue the descending sphere in the standard $1$-handle as above to the product $[0,1] \times (Y_1 \sqcup Y_2)$ along neighborhoods of the isotropic spheres in $\{1\}\times (Y_1 \sqcup Y_2)$. This $1$-handle attachment yields $(Y_1\# Y_2, \xi_1 \# \xi_2)$ where the contact structure $\xi_1 \# \xi_2$ is the same as $\xi_1$ and $\xi_2$ when restricted to the complement of the neighborhoods of the isotropic spheres (Theorem 5.1 in \cite{Weinstein_one_handle}). In fact, the proof of Theorem 5.1 in \cite{Weinstein_one_handle} shows the following.

\begin{lemma}
\label{lem:one_handle}
     Let $(Y_1 \# Y_2, \lambda_1 \# \lambda_2)$ be obtained by attaching the Weinstein one-handle defined above along the isotropic sphere in the contact manifold $(Y_1 \sqcup Y_2, \lambda_1 \sqcup \lambda_2)$. Then, the contact form $(Y_1 \# Y_2, \lambda_1 \# \lambda_2)$ differs from that on $(Y_1 \sqcup Y_2, \lambda_1 \sqcup \lambda_2)$ only on the neighborhood of the isotropic sphere where the surgery takes place. In addition, $\lambda_1 \# \lambda_2$ can be perturbed to be non-degenerate if $\lambda_1 \sqcup \lambda_2$ is non-degenerate.
\end{lemma}
\begin{proof}
    Consider the following isotropic setup 
    $$(V, \omega, x, Y, \Lambda) = ([0,1] \times (Y_1 \sqcup Y_2), d(e^t (\lambda_1 \sqcup \lambda_2)), \partial t,\{1\} \times (Y_1 \sqcup Y_2), S^0 = \{p\}\sqcup\{q\}).$$
    where $p\in Y_1$ and $q\in Y_2$ are the two points away from orbits on which we are performing the connected sum operation.
    Consider the second isotropic setup given by the standard Weinstein handle described in Section \ref{sec:weinstein}
    $$(V', \omega', x', Y', \Lambda') = (\R^4, \omega, X, \{f = -1\}, S_-)$$
    Now applying Proposition \ref{prop:isotropic_setup} gives an isomorphism of isotropic setups and results in the contact manifold $(Y_1 \# Y_2, \lambda_1 \# \lambda_2)$. The strict contactomorphism away from the surgery region is given by flowing along $\partial t$. This is because $\mathcal{L}_{\partial t} \omega = \omega$ since $\partial t$ is a Liouville vector field, so $i_{\partial t} \omega$ is preserved under flowing along $\partial t$. Now, by construction $\lambda_1 \# \lambda_2$ is nondegenerate on the complement of the surgery region. The linearized flow in (\ref{eq:flow}) shows that $h$ is also nondegenerate. For other orbits potentially created during the connected sum operation, they must intersect the boundary spheres $S^2_p\sqcup S^2_q$ of a neighborhood of $\Lambda = \{p\}\sqcup \{q\}$. Therefore, we may find a $C^\infty$-small perturbation supported in an arbitrarily small neighborhood of the spheres $S^2_p\sqcup S^2_q$ as in \cite{Fish_Siefring} so that $\lambda_1 \# \lambda_2$ is nondegenerate while maintaining the results in Theorem \ref{thm:Fish_Siefring} in the next section.
\end{proof}

\subsection{Almost complex structures}
\label{sec:almost_complex_structures}
In this subsection, we discuss the pseudo-holomorphic behaviors in the symplectization of the connected sum region $\mathbb{R} \times S^2 \subset \mathbb{R} \times (Y_1 \# Y_2)$. This has been studied in \cite{232,Fish_Siefring}. Fix a generic $\lambda_i$-adapted almost complex structure $J_i$ on $\R \times (Y_i,\lambda_i)$ and let 
$$\lambda':=\lambda_1 \# \lambda_2$$ 
be the connected sum contact form described in Lemma \ref{lem:one_handle}. Fish-Siefring found a $\lambda'$-adapted almost complex structure such that one may see a pair of pseudo-holomorphic planes asymptotic to $h$ in $\mathbb{R} \times S^2$. We continue to follow the paper \cite{Fish_Siefring}, now for information regarding the pseudo-holomorphic curves. 

\begin{theorem}\cite[Theorem 5.1]{Fish_Siefring}
\label{thm:Fish_Siefring}
Let $Y:=Y_1\sqcup Y_2$ be a three-manifold equipped with a nondegenerate contact form $\lambda$. Let $Y'$ be the connected sum manifold equipped with a nondegenerate contact form $\lambda'$ given by Lemma \ref{lem:one_handle}. Then, given any $\lambda'$-adapted $J$, there exists a $\lambda'$-adapted $J'$ agreeing with $J$ outside the surgery region such that:
\begin{enumerate}
    \item There exists a pair of embedded, disjoint $J'$-holomorphic planes $P_N$ and $P_S$ both asymptotic to $h$. 
    \item The planes $P_N$ and $P_S$ approach $h$ in \emph{opposite directions}, that is, their leading asymptotic coefficients are of opposite signs (see Section \ref{sec:asymptotic_h}).
    \item The planes $P_N$ and $P_S$ project to the northern and southern hemispheres of the connected sum sphere $S_+$ under $\pi: \R\times Y' \to Y'$. Together with $\R \times h$, their $\R$-translations foliate $\R \times S_+$.
\end{enumerate}
\end{theorem}

\begin{proof}
Consider the $1$-form 
$$\alpha = i_X \omega|_{f=1} = \frac{1}{2}xdy - \frac{1}{2}ydx + 2zdw + wdz|_{f=1}$$ 
on $\R \times S_+$ considered in Section \ref{sec:weinstein}. The coordinates in \cite{Fish_Siefring} are related to ours by the following coordinate change $\Phi$:
    \begin{align*}
        x(\rho, \phi, \theta) &= \sin \theta \cos \phi, \\
        y(\rho, \phi, \theta) &= \sin \theta \sin \phi, \\
        z(\rho, \phi, \theta) &= 4 \cos \theta, \\
        w(\rho, \phi, \theta) &= \frac{3}{8}\rho.
    \end{align*}
    Now, at the sphere $w=\rho=0$,
    \begin{align}
        \Phi^* \alpha &=\frac{1}{2}\sin^2 \theta d\phi + 3 \cos \theta d\rho,
    \end{align}
which agrees with the contact form given in \cite{Fish_Siefring} up to a constant. Therefore, the proof of Lemma 5.7 in \cite{Fish_Siefring} goes through.
\end{proof}

Now we compute the ECH index of the planes in Theorem \ref{thm:Fish_Siefring}.
\begin{lemma}
\label{lem:index_I_1}
$I(P_S) = I(P_N) = ind(P_S) = ind(P_N) = 1$.
\end{lemma}
\begin{proof}
Fix trivialization $\tau_{0}$ as in Section \ref{sec:weinstein}, then $$c_{\tau_0}(P_S) = 1$$ 
by considering the vector field $$zx \partial y - zy \partial x + (z^2-1) \partial w$$
which has one unique positive zero at $z=-1$. The zero is positive by the Poincar\'e-Hopf theorem. This vector field is in the contact structure $\ker \alpha$ restricted to $P_S$ and tangent bundle of $P_S|_h$ while being normal to the boundary, and under the Reeb flow this vector field does not rotate with respect to the contact structure restricted to $h$. Therefore,
$$ind(P_S) = -\chi(C) + 2c_{\tau_{0}}(C) + CZ_{\tau_{0}}(h) = -1+2+0 = 1.$$

Now, to compute the ECH index, we first need to compute $Q_{\tau_0}(P_S)$. Since our trivialization $\tau_0$ is the same trivialization associated to the constant section given by $\partial z$, one can see that 
$$Q_{\tau_0}(P_S) = 0.$$ 
Therefore, 
$$I(P_S) = c_{\tau_0}(P_S) + Q_{\tau_0}(P_S) + CZ^I_{\tau_0}(P_S) = 1 + 0 + 0 = 1.$$
The case of $P_N$ follows analogously.
\end{proof}

To see that all the curves are transversely cut out with the almost complex structure obtained in Theorem \ref{thm:Fish_Siefring}, we first observe that in the open set $W$ defined by $\R \times (Y_1 \# Y_2)$ minus the symplectization of the connected sum sphere $\mathbb{R}\times S^2$, we can perturb our complex structure to be generic. Therefore, all the curves except for the two $ind= 1$ planes bounding the hyperbolic orbit $h$ are transversely cut out, because they all intersect $W$. Then, we observe that for $P_N$ and $P_S$ which are embedded, we have automatic transversality (e.g. Proposition A.1 in \cite{wendl_strongly}). This is because $$ind(P_S)>c_N(P_S) = \frac{ind(P_S) -2 + 2 g +\# \Gamma_0}{2} = 0,$$
where $\# \Gamma_0$ denotes the number of asymptotic ends with even Conley-Zehnder indices. Note that the parity of the Conley-Zehnder index is independent of the choice of trivialization. The calculation is the same for $P_N$. 

\begin{remark}
Another useful way to obtain the contact connected sum is the following. First, drill two convex sutured balls $(D^2 \times I, \Gamma)$, where the suture is $\Gamma = \partial D^2 \times \{1/2\}$, from the two contact three-manifolds $Y_1$ and $Y_2$. Turn one of the concave sutured manifolds, for example, $Y_1 \backslash (D^2 \times I, \Gamma)$, into a convex sutured manifold $Y_1(1)$ as described in \cite{sutures}. This concave-to-convex process introduces a positive hyperbolic orbit. Finally, glue $Y_1(1)$ and $Y_2\backslash (D^2 \times I, \Gamma)$ together to obtain a contact connected sum $Y_1 \# Y_2$. Although this construction also results in exactly one extra positive hyperbolic orbit up to large action and yields isotopic contact structures as the Weinstein construction above, it is not a priori clear how the almost complex structure on its symplectization corresponds to that coming from the Weinstein $1$-handle attachment.
\end{remark}

\subsection{Asymptotic neighborhood of the special hyperbolic orbit \emph{h}}
\label{sec:asymptotic_h}
The asymptotic neighborhood of a Reeb orbit encodes a lot of information about the potential pseudo-holomorphic curves that could asymptote to it. In order to understand the asymptotic neighborhood of $h$, we first review the asymptotic operator associated to a Reeb orbit. For more details, see for example \cite{Siefring} and \cite{HT_gluing1}.

Let $C$ be a somewhere injective $J$-holomorphic curve with an asymptotic end at an embedded Reeb orbit $\gamma$. By rescaling the $s$ and $t$ coordinates on $\R \times Y$ near the Reeb orbit $\gamma$, we may assume that $\gamma$ has period $1$. The almost complex structure $J$ on $\xi|_\gamma$ defines a family of $2 \times 2$ matrices $J_t$ such that $J_t^2 = -1$ where $t\in \R/\Z$. The linearized Reeb flow $\Psi(t)$ along $\gamma$ gives a symplectic connection 
$$\nabla^R_t = \partial t + S_t$$ 
on the $\xi|_\gamma$, where $S_t$ is a symmetric matrix for each $t\in \R/\Z$.

\begin{definition}
\label{def:asymptotic_operator}
The asymptotic operator\footnote{Our asymptotic operator is defined with an opposite sign from others in the literature, e.g. \cite{HWZI}.} is defined as 
$$L_\gamma: = J_t \nabla^R_t: C^\infty(S^1, \gamma^* \xi) \to C^\infty(S^1, \gamma^* \xi).$$
\end{definition}
 
More specifically, let $\tau$ be a complex linear, symplectic trivialization of $\xi|_\gamma$ identifying $J_t|_\xi$ with the standard complex structures $J_0$ on $\R^2 \cong \C$. For each $t\in S^1$, define 
$$S_t:= - J_0 \frac{d \Psi(t)}{dt} \Psi^{-1}(t).$$ Then we may write 
$$L_\gamma = J_0 \frac{d}{dt} + S_t.$$
Note $L_\gamma$ is self-adjoint.

Let $\eta(t)$ be an eigenfunction of $L_\gamma$ with eigenvalue $\lambda$. Then $\eta$ solves the ODE
$$\frac{d\eta(t)}{dt} = J_0(S_t - \lambda) \eta(t),$$
and hence $\eta$ is nonvanishing, if it's nonzero. Therefore, we may define $\wind_\tau(\eta)$ to be the winding number of the loop $\eta: \R/2\pi \Z \to \C$ around zero. The above discussion generalizes to multiple covers $\gamma^d$ of a simple orbit $\gamma$. Since we will focus on the hyperbolic orbit $h$ which does not allow multiple covers as ECH generators, we refer the readers to \cite{Siefring, HT_gluing1} for further discussions. 

We recall the following useful properties of eigenvalues and winding numbers associated to an asymptotic operator.

\begin{lemma}\cite[Lemma 6.4]{Hutchings_index_inequality} and \cite[Chapter 3]{HWZII}
\label{lem:eigen_lemmas}
\begin{enumerate}
    \item If $\eta, \eta'$ are eigenfunctions of the asymptotic operator $L$ corresponding to eigenvalues $\lambda\leq \lambda'$, then $\wind(\eta)\geq \wind(\eta')$.
    \item For each winding number $w$, the space of eigenfunctions with winding number $w$ is $2$-dimensional.
    \item If $\gamma$ is a nondegenerate simple orbit, then $\wind_\tau (\phi) \leq \floor{CZ_\tau(\gamma)/2}$ for $\lambda >0$ i.e. $\lambda$ is an eigenvalue associated to an asymptotic operator at a positive end, and $\wind_\tau (\phi) \geq \ceil{CZ_\tau(\gamma)/2}$ for $\lambda < 0$ i.e. $\lambda$ is an eigenvalue associated to an asymptotic operator at a negative end.
\end{enumerate}
\end{lemma}

Now we review the asymptotic expansion of an \emph{asymptotically cylindrical pseudo-holomorphic curve} at a given nondegenenerate orbit. From now on, we focus on the case when a curve is positively asymptotic to a Reeb orbit, i.e. $s>>0$. The case for the negative asymptotic end is completely analogous. The following definition is from \cite{Siefring}.
\begin{definition}
    Let $\gamma$ be a Reeb orbit. Let $u$ be a $J$-holomorphic curve positively asymptotic to $\gamma$. Let $\widetilde{u}:[R,\infty) \times S^1 \to \gamma^*\xi$ be a smooth map satisfying $\widetilde{u}(s,t)\in \xi|_{\gamma(t)}$ for all $(s,t)\in [R, \infty) \times S^1$, where $R$ is a large real number. Then, $\widetilde{u}$ is an \emph{asymptotic representative} of $u$ if there exists a proper embedding $\psi:[R, \infty)\times S^1 \to \R \times S^1$ asymptotic to the identity, so that
    \begin{equation*}
        u(\psi(s,t)) = (s, \exp_{\gamma(t)} \widetilde{u}(s,t)) = \widetilde{\exp}_{(s, \gamma(t))}(0,\widetilde{u}(s,t))
    \end{equation*}
    for all $(s,t)\in [R, \infty) \times S^1$, where $\exp$ and $\widetilde{\exp}$ are the exponential maps associated to a Riemannian metric on $Y$ and respectively its $\R$-invariant lift to $\R \times Y$.
\end{definition}

The asymptotic representative encodes the isotopy class of a braid $\zeta$ that is the intersection of the pseudo-holomorphic curve asymptotic to $\gamma$ and a tubular neighborhood $\mu$ of $\gamma$ at $s>>0$. Hofer-Wysocki-Zehnder \cite{HWZI} showed that the map $\widetilde{u}$ can be written in the form
\begin{equation}
\label{eq:asymptotic_HWZ}
    \widetilde{u}(s,t) = e^{-\lambda s}(\eta(t) + r(s,t))
\end{equation}
for an eigenvalue $\lambda>0$ and eigenfunction $\eta(t)$ of $L_\gamma$ and an error term $r(s,t)$ converging exponentially to zero as $s\to \infty$. The thesis of Siefring analyzed the difference of two pseudo-holomorphic half-cylinders asymptotic to the same Reeb orbit and as a result generalized the asymptotic form in Equation (\ref{eq:asymptotic_HWZ}) to ``higher orders'' \cite{Siefring}.

\begin{theorem}\cite[Theorem 2.2]{Siefring}
\label{thm:Siefring_2_2}
    Let $u,v$ be $J$-holomorphic curves with a positive asymptotic end at $\gamma$ and the maps $U,V:[R,\infty) \times S^1 \to C^\infty(\gamma^*\xi)$ be their asymptotic representatives. Assume $U-V$ does not vanish identically. Then there exists a positive eigenvalue $\lambda$ of the asymptotic operator $L_\gamma$ with an eigenfunction $\eta(t)$ such that
    \begin{equation}
        U(s,t) - V(s,t) = e^{-\lambda s}(\eta(t) + r(s,t)),
    \end{equation}
    where the map $r(s,t)$ satisfies the decay estimate
    \begin{equation}
    \label{eq:error_estimate}
    |\nabla_s^i \nabla_t^j r(s,t)|\leq M_{ij} e^{-ds}
    \end{equation}
    for every $(i,j)\in \mathbb{N}^2$, where $M_{ij}$ and $d$ consist positive constants.
\end{theorem}

The above result due to Siefring implies that, for any positive integer $N$, we may expand a positive end of a pseudo-holomorphic curve by 
\begin{equation}
\widetilde{u}(s,t) = \sum_i^N e^{-\lambda_i s} a_i \eta_i(t) + o_\infty(\lambda_N),
\end{equation} 
where $a_i \in \R$ keeps track of the scalar of the eigenfunction $\eta_i(t)$ and $o_\infty$ is a function $f: [R, \infty) \times S^1 \to \mathbb{R}^2$ satisfying the decay estimate (\ref{eq:error_estimate}), c.f. Theorem 2.3 in \cite{Siefring}. We call $a_i$ the \emph{asymptotic coefficients} of the asymptotic expansion of $u$ at $\gamma$. An immediate corollary is the following.

\begin{corollary}
\label{cor:different_curves_different_coefficient}
    Let $u,v$ be distinct pseudo-holomorphic curves with a common asymptotic end at $\gamma$. Then the asymptotic coefficients of the asymptotic expansions of $u$ and $v$ at $\gamma$ differ at some term. 
\end{corollary}

To fix notations, from now on, let $(\lambda_i, \eta_i)$ be the corresponding eigenvalues and eigenvectors of the asymptotic operator $L_h$ associated to the hyperbolic orbit $h$. Fix some integer $N>>0$. Let the asymptotic expansion of $P_N$ be 
$$\sum_i^N e^{-\lambda_i s} a_i \eta_i(t) +o_\infty(\lambda_N)$$ 
and the asymptotic expansion of $P_S$ be 
$$\sum_i^N e^{-\lambda_i s} b_i \eta_i(t) +o_\infty(\lambda_N).$$ 
Notice that we use the same index $i$, and some of $a_i$ and $b_i$ could be zero. In addition, $a_0$ and $b_0$ are of opposite signs by Theorem \ref{thm:Fish_Siefring}(2).

By Lemma \ref{lem:eigen_lemmas}(3), we see that $\wind_{\tau_0} (h) \leq 0$ when the special hyperbolic orbit $h$ is at a positive end. In Lemma \ref{lem:winding_zero}, we show that this bound is achieved. However, before that, we need to introduce a few more definitions in preparation of the proof.

Specifically, let $\zeta$ be a braid in a tubular neighborhood $\mu$ of $\gamma$ determined by a pseudo-holomorphic curve $u$ asymptotic to $\gamma$ as above. Identify 
$$\mu \cong S^1 \times \mathbb{R}^2 \cong A\times (0,1)\subset \mathbb{R}^3$$ 
with respect to a trivialization $\tau$, where $A$ denotes an annulus. Let $\gamma$ be an embedded Reeb orbit such that $C$ has positive ends of multiplicities $q_1, \dots, q_n$ at $\gamma$ with total multiplicity $m$. Each end of a pseudo-holomorphic curve determines a braid component $\zeta_i$ of the braid $\zeta$ with $q_i$ strands, where $\zeta = im(u) \cap \{s\} \times \mu$ for $s>>0$ as before.

\begin{definition}
\label{lem:braid_writhe_bound}
The \emph{writhe} $w_\tau(\zeta)\in \mathbb{Z}$ is defined to be one half the signed count of crossings under the projection $A\times (0,1) \to A$.
\end{definition}

\begin{definition}
Let $\zeta_1$ and $\zeta_2$ be two disjoint braids in the neighborhood of $\gamma$. Then the \emph{linking number} $l_\tau(\zeta_1, \zeta_2)\in \mathbb{Z}$ is defined to be one half the signed count of crossings between $\zeta_1$ and $\zeta_2$ under the projection $A\times (0,1) \to A$.
\end{definition}

In light of the braid picture, the winding number of an eigenfunction $\eta_\zeta$ associated to the braid $\zeta$ around the underlying simple orbit $\gamma$ is exactly 
$$\wind_\tau(\eta_\zeta) = l_\tau(\zeta, \gamma).$$

\begin{definition}
\label{def:linking_surfaces}
Let $S$ and $S'$ be admissible representatives of $H_2(Y, \alpha, \beta)$ and $H_2(Y, \alpha', \beta')$. Then,
$$w_\tau(S):= \sum_i w_\tau(\zeta_i^+) - \sum_j w_\tau(\zeta_j^-),$$
$$\wind_\tau(S):= \sum_i \wind_\tau(\zeta_i^+) - \sum_j \wind_\tau(\zeta_j^-),$$
$$l_\tau(S,S'):= \sum_i l_\tau(\zeta_{i}^{+}, \zeta_{i}^{+}{}') - \sum_j l_\tau(\zeta_j^-, \zeta_j^{-}{}'),$$
where $\zeta_i^+$ and $\zeta_i^{+}{}'$ are braids defined by the intersections of $S$ and $S'$ with $(\{1-\epsilon\} \times Y)$, and $\zeta_j^-$ and $\zeta_j^{-}{}'$ similarly with $\{-1+\epsilon\} \times Y$.
\end{definition}

We have the following ``linking bound'' lemma that first appeared in \cite{Hutchings_index_inequality}.
\begin{lemma}\cite[Lemma 5.5(b)]{Hutchings}
\label{lem:linking_bound}
    Let $$\rho_i:=\Big\lfloor\frac{CZ_\tau(\gamma^{q_i})}{2}\Big\rfloor.$$ 
    If $i\neq j$, then $\zeta_i$ and $\zeta_j$ are disjoint and $$l_\tau(\zeta_i, \zeta_j) \leq \max(\rho_i q_j, \rho_j q_i).$$
\end{lemma}

In particular, any two braids $\zeta_i$ and $\zeta_j$ associated to the asymptotic neighborhood at $h$ have bounded linking number $l_{\tau_0}(\zeta_i, \zeta_j) \leq 0$ since $CZ_{\tau_0}(h) = 0$.

\begin{lemma}
\label{lem:winding_zero}
$\wind_{\tau_0} (P_S) = \wind_{\tau_0} (P_N) = 0$.
\end{lemma}
\begin{proof}
We prove the lemma for $P_S$. The proof for $P_N$ is exactly the same. Applying the relative adjunction formula in Proposition \ref{prop:adjunction} to $C:= P_S \sqcup P_S'$, where $P_S'$ is an $s$-translation of $P_S$, we get
$$c_{\tau_0}(C) = \chi(C)+Q_{\tau_0}(C)+w_{\tau_0}(C)-2\delta(C).$$
By the construction of almost complex structure in Theorem \ref{thm:Fish_Siefring} that admits the foliation of $\R \times S^2$ where $P_S$ is a leaf, we have that $\delta(C) = 0$. We also know that $Q_{\tau_0}(C) = 0$ since $Q_\tau$ is quadratic and $Q_{\tau_0}(P_S) = 0$. In addition, $w_{\tau_0}(P_S) = 0$. This is because the multiplicity of $h$ is one and therefore $w_{\tau_0}(\zeta_h) = 0$ where $\zeta_h$ is the braid at the intersection of $P_S$ and $\{1-\epsilon\}\cap Y$. Now we have that $$2c_{\tau_0}(P_S) = 2 + 2\wind_{\tau_0}(P_S),$$
since 
$$w_{\tau_0}(P_S \sqcup P_S') = 4 w_{\tau_0}(P_S)+ 2\wind_{\tau_0}(P_S) = 2 \wind_{\tau_0}(P_S),$$ 
where the first equality is justified in the proof of Proposition 8.4 in \cite{Hutchings_index_inequality}. Therefore,
$$\wind_{\tau_0}(P_S) = 0.$$
\end{proof}

The above lemma is needed for our specific case when the almost complex structure on $\R \times S^2$ is not generic as constructed in Theorem \ref{thm:Fish_Siefring}. When we have generic almost complex structure, the following proposition says that the extremal bound of the winding number is always achieved.

\begin{proposition}\cite[Proposition 3.2]{HT_gluing2} 
\label{prop:maxwinding}
If the symplectization-adapted almost complex structure $J$ on $\R \times Y$ is generic, then for any Fredholm index $1$, connected, non-multiply-covered $J$-holomorphic curve $C$ having a positive end at $\gamma$, the winding number of the leading eigenfunction of the asymptotic expansion of $C$ achieves the equality in Lemma \ref{lem:eigen_lemmas}(3).
\end{proposition}

\section{Correspondence of Reeb orbits and filtered ECH complex}
\label{sec:correspondence}
In this section, we discuss how to correspond the Reeb orbits in the closed manifolds $Y_1 \# Y_2$ to the ones in $Y_1$ and $Y_2$. We  show that up to a sufficiently large action $L$, we may ignore orbits crossing the connected sum sphere. This allows an identification on the vector space level of a filtered ECH complex of the connected sum and a filtered mapping cone complex.

\subsection{The mapping cone}
\label{sec:mapping_cone}
Given a chain map between two chain complexes, one can form its \emph{mapping cone}. In the following, we review this construction in our setting. We define 
$$C_o := ECC(Y_1 \sqcup Y_2, \lambda_1 \sqcup \lambda_2) = ECC(Y_1, \lambda_1)\otimes_{\F} ECC(Y_2, \lambda_2).$$

\begin{definition}
    Given the chain map $\varphi: C_o \to C_o$ defined by 
    $$\varphi:= U_1 \otimes id + id \otimes U_2,$$
    one can form the \emph{mapping cone} $Cone(\varphi)$ by the following. As a vector space, 
    $$Cone(\varphi) := C_o \oplus C_o[-1],$$ 
    where $C_o[-1]^*:=C_o^{*-1}$. The differential of $Cone(\varphi)$ is given by
\begin{equation*}
\partial_{cone} := 
\begin{pmatrix}
\partial_1 \otimes id + id \otimes \partial_2 & 0 \\
\varphi & \partial_1 \otimes id + id \otimes \partial_2
\end{pmatrix}.
\end{equation*}
\end{definition}
Let $C_h:= C_o \otimes h$ denote the collection of all orbit sets in $Y_1\sqcup Y_2$ concatenating with the special hyperbolic orbit $h$. By Theorem \ref{thm:Fish_Siefring} and Lemma \ref{lem:index_I_1}, we know that $h$ bounds $I=1$ planes $P_N$ and $P_S$. Therefore, one can think of $h$ as an element that increases the relative grading by one. We can then identify $C_o[-1]$ with $C_h$ and obtain that $$Cone(\varphi) = C_0 \oplus C_h$$ 
as vector spaces. Furthermore, we define a filtered mapping cone complex.
\begin{definition}
\label{def:Cone_L}
    Given a real number $L$, we define 
    $$Cone^L(\varphi):= C_o^L \oplus C_h^L = C_o^L \oplus {C_o}^{L-\mathcal{A}(h)}.$$
\end{definition}

One may check that $\partial_{cone}$ preserves $Cone^L(\varphi)$, since both the differential and the U map decrease the symplectic action by Stokes' theorem and the definition of a $\lambda$-adapted almost complex structure. Therefore, $Cone^L(\varphi)$ is a subcomplex of $Cone(\varphi)$. In addition, we observe that the direct limit
$$\lim_{L\to\infty}Cone^L(\varphi) = Cone(\varphi)$$
given by the obvious inclusion maps.

\subsection{Filtered ECH of a connected sum}
\label{sec:filtered_ech_of_a_connected_sum}
We use filtered ECH to ignore potential orbits of large symplectic actions formed during the connected sum procedure, which cross the connected sum $S^2$. The following lemma is folklore.

\begin{lemma}
\label{lem:large_energy}
Let $L>0$. Let $p$ be a point which is not on a Reeb orbit of action $< L$. Then there exists a radius $R(L)$ such that any Reeb trajectory which starts on the ball $B$ of radius $R(L)$ centered at $p$, leaves $B$ and returns to $B$, has symplectic action greater or equal $L$.
\end{lemma}

\begin{proof}
Suppose not. We take a sequence of balls $B_n$ around $p$ of radius converging to $0$, and Reeb trajectories $\gamma_n$ of action less than $L$ which start on $B_n$, leave $B_n$, and return to $B_n$. We can then pass to a subsequence so that the Reeb trajectories $\gamma_n$ converge to a Reeb orbit $\gamma$ of action less than $L$ which passes through $p$. Contradiction.
\end{proof}

The above lemma tells us that up to an action $L$, we may ignore orbits that pass through the connected sum region other than the special hyperbolic orbit $h$. More precisely, using Lemma \ref{lem:large_energy}, we may find a strictly decreasing smooth function $R: \R \to \R$ such that $\lim_{L\to \infty} R(L) = 0$ and 
\begin{equation}
\label{eq:vector_space}
    ECC^L(Y_1 \#_{R(L)} Y_2, \alpha_{R(L)}) \cong C_o^L \oplus C_h^L= Cone^L
\end{equation}
as vector spaces, where $\alpha_{R(L)} := \lambda_1 \#_{R(L)} \lambda_2$ denotes the contact form on $Y_1 \#_{R(L)} Y_2$ where the function $f$ in Section \ref{sec:weinstein} is scaled by $R(L)$. Note that any orbit intersecting a connected sum sphere of radius $r$ would also intersect a connected sum sphere of radius $r'>r$, so we may assume $R(L)$ in Lemma \ref{lem:large_energy} is non-increasing. Furthermore, we can modify $R(L)$ to be smooth and strictly decreasing. Therefore, by (\ref{eq:vector_space}), we obtain an isomorphism on the level of vector spaces of the following theorem.

\begin{proposition}[Proposition \ref{prop:main_filtered_intro}]
\label{prop:main_filtered}
Given two closed connected contact three-manifolds $(Y_1, \lambda_1)$ and $(Y_2, \lambda_2)$ with nondegenerate contact forms $\lambda_i$. Then there exists a strictly decreasing function $R: \R \to \R$ with $$\lim_{L\to \infty} R(L) = 0,$$
such that there is a chain homotopy equivalence 
$$f: ECC^L(Y_1 \#_{R(L)} Y_2, \alpha_{R(L)}) \to Cone^L(U_1 \otimes id + id \otimes U_2).$$
\end{proposition}
The main goal of Section \ref{sec:differentials} and Section \ref{sec:chain_homotopy} is to prove Proposition \ref{prop:main_filtered} by studying the connected sum differentials. Then, in Section \ref{sec:direct_limit}, we will pass the filtered ECH chain complexes $ECC^L(Y_1 \#_{R(L)} Y_2, \alpha_{R(L)})$ to a direct limit, where we consider smaller and smaller connected sum spheres as we let the symplectic action $L$ go to infinity, to obtain the main theorem.

\section{The differentials for ECH of a connected sum}
\label{sec:differentials}
There are four types of curves to consider for the differential of $ECC^L(Y_1 \#_{R(L)} Y_2, \alpha_{R(L)})$: 
\begin{itemize}
    \item $I = 1$ curves in $\R \times (Y_1 \#_{R(L)} Y_2)$ with positive ends in $C_o^L$ and negative ends in $C_o^L$, denoted as $\partial_{oo}$; 
    \item $I = 1$ curves in $\R \times (Y_1 \#_{R(L)} Y_2)$ with positive ends in $C_h^L$ and negative ends in $C_o^L$, denoted as $\partial_{oh}$; 
    \item $I = 1$ curves in $\R \times (Y_1 \#_{R(L)} Y_2)$ with positive ends in $C_o^L$ and negative ends in $C_h^L$, denoted as $\partial_{ho}$; 
    \item $I = 1$ curves in $\R \times (Y_1 \#_{R(L)} Y_2)$ with positive ends in $C_h^L$ and negative ends in $C_h^L$, denoted as $\partial_{hh}$. 
\end{itemize}
We organize this as
$$\partial_\# = \begin{pmatrix}
\partial_{oo} & \partial_{oh}\\
\partial_{ho} & \partial_{hh}
\end{pmatrix}.$$

In this section, we describe how the differentials $\partial_{oo}, \partial_{oh}$ and $\partial_{hh}$ behave in terms of the original differentials in $ECC^L(Y_1,\lambda_1)$ and $ECC^L(Y_2,\lambda_2)$. In particular, we show that the connected sum differentials $\partial_{oo}$ and $\partial_{hh}$ are identified with the original differentials on $Y_1$ and $Y_2$, considering the additional trivial cylinder over $h$. The differential $\partial_{oh}$ is reminiscent of the linearized contact homology case \cite{BVK}, where the special hyperbolic orbit can be assumed to have small action and hence cannot admit a curve positively asymptotic to it while having negative asymptotics. In our case where a curve may have multiple positive asymptotic ends, we need to study the asymptotic behavior near the neighborhood of the hyperbolic orbit $h$ more closely. The differential $\partial_{ho}$ involves the original $U$ maps in $ECC^L(Y_1,\lambda_1)$ and $ECC^L(Y_2,\lambda_2)$ and will be discussed in Section \ref{sec:partial_ho}.

\subsection{No-crossing lemma}
\label{sec:no_crossing}
This lemma appears in the draft of the connected sum formula for linearized contact homology \cite{BVK} and we give essentially the same proof here using the relative intersection number $Q_\tau$ and the asymptotic linking number $l_\tau$. 

\begin{lemma}[No-crossing]
\label{lem:no_crossing}
Let $C$ be a connected, embedded $J$-holomorphic curve in the symplectization $\R \times (Y_1 \#_{R(L)} Y_2, \alpha_{R(L)})$ with ends that do not intersect the connected sum $S^2$. Then the image of $C$ cannot have ends in both $Y_1$ and $Y_2$ away from the connected sum region.
\end{lemma}
\begin{proof}
Suppose for the contradiction that we have a curve $C$ with positive ends at $\alpha$ and negative ends at $\beta$, such that $\alpha \cup \beta$ lie in both $Y_1$ and $Y_2$ away from the Darboux balls used for the connected sum operation. Then $C$ must intersect $\R \times S_+$. In particular, $C$ must intersect one of the holomorphic planes that are the translations of $P_N$ the $P_S$ given in Theorem \ref{thm:Fish_Siefring} or the trivial cylinder over $h$. We may assume without loss of generality that $C$ intersects $P_N$ or $C$ intersects the trivial cylinder over $h$. Therefore, we have $\# (\dot{C} \cap \dot{P_N}) \geq 1$ or $\# (\dot{C} \cap \dot{(\R \times h)}) \geq 1$. 

Suppose $\# (\dot{C} \cap \dot{P_N}) \geq 1$. Recall that 
$$Q_{\tau}(C,P_N) = \#(\dot{S} \cap \dot{S}') - l_{\tau}(S,S'),$$
where $S$ and $S'$ are any admissible representatives of $[C]$ and $[P_N]$, and the interiors $\dot{S}$ and $\dot{S'}$ are transverse and do not intersect near the boundary. We have that 
$$l_{\tau}(S,S') = 0$$ 
since by assumptions the ends of $C$ do not overlap with the end $h$ of $P_N$. 

Now it suffices to show that $$Q_{\tau}(C,P_N) =0.$$ 
By a usual transversality argument, one can find another representative of $[C]\in H_2(Y, \alpha, \beta)$, which is still denoted as $S$, such that the intersections of $S$ and $S'$ miss a neighborhood of the north pole on $P_N$. Therefore, we may define a homotopy $S_{t\in [0,1]}$ moving all possible interior intersections of $S$ and $S'$ to be on $P_S$. Therefore, 
$$\#(\dot{S} \cap \dot{S'}) = \#(\dot{S_1} \cap \dot{S'}) = 0,$$ 
so $Q_{\tau}(C,P_N) = 0$. The arguments for $Q_{\tau}(C, \R \times h) = 0$ is entirely analogous.
\end{proof}

\subsection{\texorpdfstring{$\partial_{oo}$}{diff oo}: differential from orbit sets without \emph{h} to orbit sets without \emph{h}}
The discussions in Section \ref{sec:filtered_ech_of_a_connected_sum} shows that, by using filtered ECH up to filtration $L$, we may focus only on orbits of actions less than $L$ in $(Y_1 \#_{R(L)} Y_2, \alpha_{R(L)})$ and all such orbits are away from the connected sum $S^2$ except for the special hyperbolic orbit $h$. We then use the no-crossing lemma to show that the part of the connected sum differential not approaching the hyperbolic orbit $h$ at all is the same as the sum of the original differential $\partial_i$ on $ECC^L(Y_i, \lambda_1)$:
\begin{lemma}
\label{lem:L_oo}
We have that the ECH differential 
$$\partial_{oo} = \partial_1 \otimes id + id \otimes \partial_2$$ 
when restricted to the subcomplex $ECC^L(Y_1 \#_{R(L)} Y_2, \alpha_{R(L)})$.
\end{lemma}
\begin{proof}
By Lemma \ref{lem:large_energy}, orbits that intersect the connected sum $S^2$ have energy $\geq L$, except for the special hyperbolic orbit $h$. To study $\partial_{oo}$, we may focus on the connected, embedded, $I=1$ curves given by Theorem \ref{thm:low_index}. Then by the Lemma \ref{lem:no_crossing}, the curves counted by $\partial_{oo}$ are exactly those that do not pass through $\mathbb{R} \times S^2$. These are exactly the $I=1$ curves either completely in $\mathbb{R}\times Y_1$ or completely in $\mathbb{R}\times Y_2$, i.e. those counted by $\partial_1 \otimes id + id \otimes \partial_2$.
\end{proof}

\subsection{\texorpdfstring{$\partial_{oh}$}{diff oh}: differential from orbit sets with \emph{h} to orbit sets without \emph{h}}
The heuristics come from ``local energy'' of curves near the asymptotic end at the special hyperbolic orbit $h$. In this subsection, we eliminate all curves that have a positive asymptotic end at $h$ and have no negative asymptotic end at $h$, other than the two planes $P_S$ and $P_N$ and their translations bounding $h$ as described in Section \ref{sec:almost_complex_structures}.

\begin{definition}
A $J$-holomorphic curve $C$ \emph{approaches $h$ from the south (resp. north)} if its leading asymptotic coefficient has the same sign as that of $P_S$ (resp. $P_N$). 
\end{definition}

\begin{proposition}
\label{prop:plane}
Let $J$ be generic in the open set $\mathbb{R}\times (Y_1\#_{R(L)}Y_2)\setminus \mathbb{R} \times S^2$. Any embedded, Fredholm index $1$, connected curve with positive ends in $C_h$ and negative ends in $C_o$ is either $P_S$ or $P_N$, or a translation of them in the $\partial s$ direction.
\end{proposition}
\begin{proof}
Suppose $C$ is an embedded, Fredholm index $1$, connected $J$-holomorphic curve with a positive asymptotic end at $h$. In particular, $C$ is not a trivial cylinder by Theorem \ref{thm:low_index}. Suppose $C$ is not a translation of $P_S$ or $P_N$. In particular, 
$$C \cap (\mathbb{R}\times (Y_1\#_{R(L)}Y_2)\setminus \mathbb{R} \times S^2) \neq \emptyset,$$
so $J$ is generic for $C$. By Lemma \ref{lem:eigen_lemmas}(3) and Proposition \ref{prop:maxwinding}, we know that the leading eigenfunction $\phi$ of $C$ at $h$ has
$$\wind_{\tau_0}(\phi) =\floor{CZ_{\tau_0}(h)/2} = 0.$$ 
Notice that this is a calculation that holds for any $C$ satisfying the assumptions.
 
Without loss of generality, assume that $C$ approaches $h$ from the south. In particular, the leading asymptotic coefficient of $C$ has the same sign as that of $P_S$. The proof for $C$ approaching from the north is exactly the same. Now by the definition of $Q_\tau$ and intersection positivity, we have that
\begin{equation}
\label{eq:Q_tau_link_tau}
    Q_{\tau_0}(C, P_S) + l_{\tau_0}(C, P_S) =  \#(\dot{C} \cap \dot{P_S}) \geq 0,
\end{equation}
since both $C$ and $P_S$ are admissible representatives of their respective relative second homology classes.

Note that the leading eigenvalues of $C$ and $P_S$ both achieve the minimal possible eigenvalue of $L_h$ by Lemma \ref{lem:winding_zero} for curves completely inside the $\R \times S^2$ region, i.e. $P_S$ and $P_N$, and by Lemma \ref{lem:eigen_lemmas}(3) and Proposition \ref{prop:maxwinding} for curves that intersect the $\R \times (Y_1 \#_{R(L)} Y_2) \setminus \R \times S^2$ region hence they are generic so they also achieve minimal winding numbers. Let the asymptotic expansions of $C$ and $P_S$ at $h$ be:
\begin{equation}
    \widetilde{C} = \sum_i^N e^{-\lambda_i s} c_i \eta_i +o_\infty(\lambda_N)
\end{equation}

\begin{equation}
    \widetilde{P_S} = \sum_i^N e^{-\lambda_i s} b_i \eta_i +o_\infty(\lambda_N)
\end{equation}
for some integer $N>>0$, where $(\lambda_i, \eta_i)$ are the corresponding eigenvalues and eigenfunctions of the asymptotic operator $L_h$ as discussed in Section \ref{sec:asymptotic_h}. From the above discussions on minimal winding numbers, we have that $c_0 \neq 0$ and $b_0 \neq 0$. In addition, by the assumption that $C$ approaches $h$ from the south, we know that $c_0$ and $b_0$ have the same sign. Now we have two possibilities.

If $c_0 = b_0$, then we may find the first $i>0$ such that $c_i \neq b_i$ by Corollary \ref{cor:different_curves_different_coefficient} and the assumption that $C \not\equiv P_S$. Now by Lemma \ref{lem:eigen_lemmas}(3), we know that the winding number that is $0$ is achieved by the corresponding eigenfunctions of both the maximal negative and minimal positive eigenvalues. Since $\lambda_0<\lambda_i$, by Lemma \ref{lem:eigen_lemmas}(1), we have that $\wind(\eta_{i}) \leq \wind(\eta_{0})$. Therefore, now by Lemma \ref{lem:eigen_lemmas}(2), $\wind(\eta_{i}) < \wind(\eta_{0}) = 0$. Therefore, we have that $$l_{\tau_0}(C, P_S) \leq \wind(\eta_{i}) \leq -1$$ 
as in the proof of Lemma 6.9 in \cite{Hutchings_index_inequality}. 

If $c_0 \neq b_0$, then we can shift $P_S$ in the $\partial s$ direction to some $s'\neq s$ corresponding to the curve $P_{S'}$ so that $c_0 e^{-\lambda_0 s} = b_0 e^{-\lambda_0 s'}$, since $c_0$ and $b_0$ have the same sign. Then some higher order $i>1$ term will contribute to at least an additional $-1$ to $\wind(\eta_0)$ by using the same arguments as in the above paragraph. Therefore, in this case we also have 
$$l_{\tau_0}(C, P_{S'}) \leq -1.$$

Recall that $Q_{\tau_0}(C, P_S)$ is well-defined under relative second homology class. Now by the inequality in (\ref{eq:Q_tau_link_tau}), it suffices to show that $Q_{\tau_0}(C, P_S) =0$. Combined with the fact that the connected sum $S^2$ is null-homologous, it suffices to show $Q_{\tau_0}(C, P_N) = 0$ by (\ref{eq:change_of_homology_class_for_Q_tau}).

By assumption, $C$ and $P_N$ have their leading asymptotic eigenfunctions of opposite signs. Therefore, $l_{\tau_0}(C, P_N) = 0$. We then have that 
$$Q_{\tau_0}(C, P_N) = \#(\dot{C} \cap \dot{P_N}) - l_{\tau_0}(C, P_N)=  \#(\dot{C} \cap \dot{P_N}).$$ 
In particular, $\#(\dot{C} \cap \dot{P_N})$ is independent of the admissible representative $\Sigma$ of $[C]$, since $Q_{\tau_0}$ is. Now, by a usual transversality argument, we may find an admissible representative $\Sigma$ of $[C]$ such that $\Sigma$ does not pass through a neighborhood of the north pole of the connected sum $S^2$. Therefore, we may define a homotopy $\Sigma_{t\in [0,1]}$ moving all possible interior intersections of $\Sigma$ and $P_N$ to be on $P_S$. Therefore, 
$$\#(\dot{C} \cap \dot{P_N}) = \#(\dot{\Sigma_1} \cap \dot{P_N}) = 0,$$
so $Q_{\tau_0}(C, P_N) = 0$. This gives a contradiction to (\ref{eq:Q_tau_link_tau}).
\end{proof}

\begin{corollary}
\label{cor:L_oh}
We have that the ECH differential
$\partial_{oh} = 0$
when restricted to the subcomplex $ECC^L(Y_1 \#_{R(L)} Y_2, \alpha_{R(L)})$.
\end{corollary}
\begin{proof}
By Lemma \ref{lem:no_crossing}, we may restrict our attention to orbits that lie in $C_o^L \oplus C_h^L$. By Proposition \ref{prop:plane}, the only ECH index $1$ curves that have a positive asymptotic end at $h$ and no negative asymptotic end at $h$ are $P_S$ and $P_N$ up to $\R$-translation. This immediately means that $\partial_{oh} = 0$ on $ECC^L(Y_1 \#_{R(L)} Y_2,\alpha_{R(L)})$, since we are working over $\F$.
\end{proof}

\subsection{\texorpdfstring{$\partial_{hh}$}{diff hh}: differential from orbit sets with \emph{h} to orbit sets with \emph{h}}
In this subsection, we show that when restricted to $ECC^L(Y_1 \#_{R(L)} Y_2, \alpha_{R(L)})$, any ECH index $1$ current that have both a positive asymptotic end and a negative asymptotic end at $h$ must separate out a trivial cylinder over $h$.
\begin{proposition}
\label{prop:trivcyl}
Let $J$ be generic in the open set $\mathbb{R}\times (Y_1\#_{R(L)}Y_2)\setminus \mathbb{R} \times S^2$. Any embedded, Fredholm index $1$ curve with positive ends in $C_h$ and negative ends in $C_h$ must contain a trivial cylinder over $h$.
\end{proposition}
\begin{proof}
The proof is exactly the same as that of Proposition \ref{prop:plane}. Consider an embedded, Fredholm index $1$, connected holomorphic curve $C$ with a positive end at $h$ and a negative end at $h$. In particular, Proposition \ref{prop:maxwinding} still holds. Now consider again the equation 
$$Q_{\tau_0}(C, P_S) + l_{\tau_0}(C, P_S) =  \#(\dot{C} \cap \dot{P_S}).$$
As in the proof of Proposition \ref{prop:plane}, the left hand side is strictly negative while by positivity of intersections, the right hand side is greater or equal to zero. Therefore, the index $1$ curve $C$ was disconnected. In order for $C$ to have a positive end at $h$, we know the only possibility is for $C$ to contain a trivial cylinder over $h$ or one of the $P_S$ and $P_N$ planes by Proposition \ref{prop:plane}. However, the latter case can be ruled out since both $P_S$ and $P_N$ are of Fredholm index $1$ and the Fredholm index is additive.
\end{proof}

\begin{corollary}
\label{cor:L_hh}
We have that the ECH differential 
$$\partial_{hh} = h \partial_{oo} \frac{1}{h}$$ 
when restricted to the subcomplex $ECC^L(Y_1 \#_{R(L)} Y_2, \alpha_{R(L)})$.
\end{corollary}
\begin{proof}
    By Lemma \ref{lem:no_crossing}, we may restrict our attention to orbits that lie in $C_o^L \oplus C_h^L$. By Proposition \ref{prop:trivcyl}, we know that any current $C$ counted by $\partial_{hh}$ must contain a trivial cylinder over $h$, which is of ECH index $0$. The remaining components of $C$ contribute exactly to currents counted by $\partial_{oo}$.
\end{proof}

\section{The chain homotopy equivalence}
\label{sec:chain_homotopy}
In this section, we discuss the last piece of the differential in the connected sum, i.e. $\partial_{ho}$, which is the differential from orbit sets without \emph{h} to orbit sets with \emph{h}. This uses a similar chain homotopy defined in the proof to show that the ECH $U$ map is independent of the change of base points (Section 2.5 in \cite{HT_weinstein}). Roughly speaking, the differential $\partial_{ho}$ in the connected sum corresponds to $U$ maps in the original contact three-manifolds $Y_1$ and $Y_2$, up to a ``chain homotopy'' $K$ discussed in Section \ref{sec:partial_ho}. This map $K$ will in turn help us construct the chain homotopy equivalence between the filtered ECH chain complex of the connected sum and the filtered mapping cone in Section \ref{sec:proof_of_main}.

\subsection{A ``chain homotopy'' induced by the change of base point}
\label{sec:partial_ho}
First, we need a lemma analogous to Lemma \ref{lem:L_oo} for the $U$ maps. Let $U_{\#, z_i}$ denote the U map in $\mathbb{R}\times (Y_1 \#_{R(L)} Y_2)$ with base point $z_i$ contained in $Y_i$ away from the connected sum region. A similar proof to that of Lemma \ref{lem:L_oo} gives the following lemma, since $I=2$ curves with asymptotic ends on ECH generators are also embedded by Theorem \ref{thm:low_index} and satisfy the no-crossing lemma.
\begin{lemma}
\label{lem:U_oo}
The moduli space of curves counted by $U_{\#, z_i}$ with asymptotic ends in $C_o$ is exactly the same as the moduli space of curves counted by $I = 2$ curves in $\R \times Y_i$ passing through the point $z_i$, for $i\in \{1,2\}$. In particular, 
$$U_1 \otimes id + id \otimes U_2 = (U_{\#, z_1})_{oo} + (U_{\#, z_2})_{oo}$$ 
when restricted to the filtered subcomplex $ECC^L(Y_1 \#_{R(L)} Y_2, \alpha_{R(L)})$.
\end{lemma}

There are very strong restrictions on which curves could have a positive end at $h$, as shown in the previous section. When we require such curves to pass through a point $p$ in $P_N$, this naturally gives more restrictions to the curves. In fact, in the following, we obtain that the only curve with a positive end at $h$ that passes through $p$ is $P_N$. 

\begin{lemma}
\label{lem:curvethrup}
Let $J$ be generic in the open set $\mathbb{R}\times (Y_1\#_{R(L)} Y_2)\setminus \mathbb{R} \times S^2$. Let $C$ be an $I=1$ curve that has positive ends in $C_h$, negative ends in $C_o\oplus C_h$, and passes through a given point $p\in P_N$. Then $C \equiv P_N$.

\end{lemma}
\begin{proof}
This is a combination of Proposition \ref{prop:plane} and Proposition \ref{prop:trivcyl}.
\end{proof}

\begin{remark}
We can also pick $p$ on $P_S$. Note that $p$ is not generic in the sense that it lies on a Fredholm index $1$ curve.
\end{remark}

\begin{definition}
Let $\mathcal{M}_{1}(\alpha; \beta; p)$ denote the moduli space of $I = 1$ curves in $\R \times (Y_1 \#_{R(L)} Y_2)$ with positive asymptotic ends at the orbit set $\alpha$ and negative asymptotic ends at orbit set $\beta$, passing through $\{0\} \times A$, where $A:=[z_1, p) \cup (p, z_2]$ is required to not pass through $P_S$ or any Reeb orbit, which is possible since there are only countably many Reeb orbits. See Figure \ref{fig:basepointp}.
\end{definition}

\begin{lemma}
Let $J$ be generic in the open set $\mathbb{R}\times (Y_1\#_{R(L)} Y_2)\setminus \mathbb{R} \times S^2$. Then $\mathcal{M}_1(\alpha; \beta; p)$ is a compact $0$-dimensional manifold.
\end{lemma}
\begin{proof}
Consider the moduli space $\widehat{\mathcal{M}}_1(\alpha; \beta; p)$ consisting of $I = 1$ curves in $\R \times (Y_1 \#_{R(L)} Y_2)$ with positive asymptotic ends at the orbit set $\alpha$ and negative asymptotic ends at the orbit set $\beta$, passing through $\{0\} \times (A \cup \{p\})$.
Observe that the moduli space $\mathcal{M}_1(\alpha; \beta; p)$ is the same as $\widehat{\mathcal{M}}_1(\alpha; \beta; p)$, except we are subtracting the holomorphic plane $P_N$ by Lemma \ref{lem:curvethrup}. Since $\widehat{\mathcal{M}}_1(\alpha; \beta; p)$ is a compact $0$-dimensional manifold (Section 2.5 of \cite{HT_weinstein}), we have that $\mathcal{M}_1(\alpha; \beta; p)$ is also a compact $0$-dimensional manifold.
\end{proof}

In fact, by relating to the change of base point homotopy in the setting of \cite{HT_weinstein}, we obtain more information. First, we define the following related map.

\begin{definition}
Let $p$ be a point on $P_N$. Given an orbit set $\alpha$, the linear map $K$ is defined as
$$K \alpha:= \sum_{\beta} \sum_{C\in \mathcal{M}_{1}(\alpha; \beta; p)} \beta,$$
where $\beta$ ranges over all orbit sets.
\end{definition}

\begin{lemma}
\label{lem:K_oh}
We have that $K_{oh} = 0$ when restricted to the subcomplex $ECC^L(Y_1 \#_{R(L)} Y_2, \alpha_{R(L)})$.
\end{lemma}
\begin{proof}
Recall again that we may restrict our attention to $C_o^L \oplus C_h^L$ by the no-crossing Lemma \ref{lem:no_crossing}. Then we apply Proposition \ref{prop:plane} and the definition of the map $K$. In particular, $P_S$ also does not intersect the arc $A$, so there is no curve counted by $K_{oh}$.
\end{proof}

\begin{figure}
    \centering
    %% Creator: Inkscape 1.1.2 (0a00cf5339, 2022-02-04), www.inkscape.org
%% Accompanies image file 'basepointp_svg-tex.pdf' (pdf, eps, ps)
\begingroup%
  \makeatletter%
  \providecommand\color[2][]{%
    \errmessage{(Inkscape) Color is used for the text in Inkscape, but the package 'color.sty' is not loaded}%
    \renewcommand\color[2][]{}%
  }%
  \providecommand\transparent[1]{%
    \errmessage{(Inkscape) Transparency is used (non-zero) for the text in Inkscape, but the package 'transparent.sty' is not loaded}%
    \renewcommand\transparent[1]{}%
  }%
  \providecommand\rotatebox[2]{#2}%
  \newcommand*\fsize{\dimexpr\f@size pt\relax}%
  \newcommand*\lineheight[1]{\fontsize{\fsize}{#1\fsize}\selectfont}%
  \ifx\svgwidth\undefined%
    \setlength{\unitlength}{156.65027344bp}%
    \ifx\svgscale\undefined%
      \relax%
    \else%
      \setlength{\unitlength}{\unitlength * \real{\svgscale}}%
    \fi%
  \else%
    \setlength{\unitlength}{\svgwidth}%
  \fi%
  \global\let\svgwidth\undefined%
  \global\let\svgscale\undefined%
  \makeatother%
  \begin{picture}(1,0.58654543)%
    \lineheight{1}%
    \setlength\tabcolsep{0pt}%
    \put(0,0){\includegraphics[width=\unitlength,page=1]{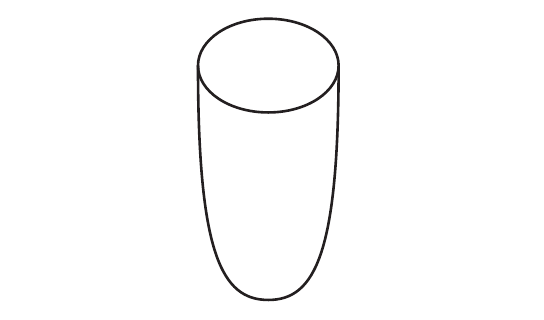}}%
    \put(0.62676239,0.50055771){\color[rgb]{0.1372549,0.12156863,0.1254902}\makebox(0,0)[lt]{\lineheight{1.25}\smash{\begin{tabular}[t]{l}h\end{tabular}}}}%
    \put(0,0){\includegraphics[width=\unitlength,page=2]{basepointp_svg-tex.pdf}}%
    \put(0.49328034,0.20745257){\color[rgb]{0.1372549,0.12156863,0.1254902}\makebox(0,0)[lt]{\lineheight{1.25}\smash{\begin{tabular}[t]{l}p\end{tabular}}}}%
    \put(0.59908928,0.09264268){\color[rgb]{0.1372549,0.12156863,0.1254902}\makebox(0,0)[lt]{\lineheight{1.25}\smash{\begin{tabular}[t]{l}$P_N$\end{tabular}}}}%
    \put(0.10035093,0.12778466){\color[rgb]{0.1372549,0.12156863,0.1254902}\makebox(0,0)[lt]{\lineheight{1.25}\smash{\begin{tabular}[t]{l}$z_1$\end{tabular}}}}%
    \put(0.8362258,0.23818983){\color[rgb]{0.1372549,0.12156863,0.1254902}\makebox(0,0)[lt]{\lineheight{1.25}\smash{\begin{tabular}[t]{l}$z_2$\end{tabular}}}}%
    \put(0,0){\includegraphics[width=\unitlength,page=3]{basepointp_svg-tex.pdf}}%
  \end{picture}%
\endgroup%

    \caption{The base point $p$ on the pseudo-holomorphic plane $P_N$}
    \label{fig:basepointp}
\end{figure}

\begin{lemma}
\label{lem:umaphomotopy}
When restricted to the subcomplex $ECC^L(Y_1 \#_{R(L)} Y_2, \alpha_{R(L)})$, we have the equation 
\begin{equation}
\label{eq:U_homotopy}
    U_1 \otimes id + id \otimes U_2 + \frac{1}{h} \partial_{ho} = \partial_{oo} K_{oo} + K_{oo}\partial_{oo}.
\end{equation}
\end{lemma}
\begin{proof}
The chain homotopy $\widehat{K}$ induced by the change of base points in \cite{HT_weinstein} gives us that 
\begin{equation}
\label{eq:HT_chain_homotopy}
    (U_{\#, z_1})_{oo} + (U_{\#, z_2})_{oo}= \partial_{oo} \widehat{K}_{oo} + \widehat{K}_{oo}\partial_{oo} + \partial_{oh}\widehat{K}_{ho} + \widehat{K}_{oh}\partial_{ho},
\end{equation}
where $\widehat{K}$ is the degree $-1$ map defined by counting $I = 1$ curves passing through the path $\{0\} \times (A \cup \{p\})$ as above. Restricting (\ref{eq:HT_chain_homotopy}) to appropriate ends up with action $< L$ and applying Lemma \ref{lem:U_oo} gives us that
\begin{equation}
\label{eq:HT_filtered_chain_homotopy}
    U_1 \otimes id + id \otimes U_2= \partial_{oo} \widehat{K}_{oo} + \widehat{K}_{oo}\partial_{oo} + \partial_{oh}\widehat{K}_{ho} + \widehat{K}_{oh}\partial_{ho}.
\end{equation}

Recall by Lemma \ref{lem:large_energy}, we have that $ECC^L(Y_1 \#_{R(L)} Y_2, \alpha_{R(L)}) \cong C_o^L \oplus C_h^L = Cone^L$ as vector spaces as in (\ref{eq:vector_space}). By the definition of $K$ and Lemma \ref{lem:no_crossing}, we immediately have that $K_{oo} = \widehat{K}_{oo}$ on the filtered complex $Cone^L$. By Lemma \ref{lem:curvethrup}, we have that 
$$K_{oh} = \widehat{K}_{oh} - \frac{1}{h}$$
on $Cone^L$. Moreover, by Lemma \ref{lem:K_oh}, we have that $K_{oh} = 0$ hence 
$$\widehat{K}_{oh} = \frac{1}{h}$$ on $Cone^L$. Recall $\partial_{oh} = 0$ by Corollary \ref{cor:L_oh} on $Cone^L$. Therefore, we obtain Equation (\ref{eq:U_homotopy}) on $Cone^L = ECC^L(Y_1 \#_{R(L)} Y_2, \alpha_{R(L)})$.
\end{proof}

\begin{remark}
There is another proof of Lemma \ref{lem:umaphomotopy} by directly analyzing the possible buildings in the breaking when one considers SFT compactness as following. Consider 
$$\mathcal{M}_2^L: = \{I=2 \text{ curves passing through } [z_1, p) \cup (p, z_2] \text{ with ends on orbit sets of actions} < L\},$$ 
which is compact because the corresponding moduli space with curves passing through $p$ is compact (see Section 2.5 in\cite{HT_weinstein}). We consider the boundary of $\mathcal{M}_2^L$. One may show that gluing the two $I=1$ levels gives back a curve passing through the same side of the connected sum $S^2$. Then, the boundary of $\mathcal{M}_2^L$ constitutes buildings $\partial K$, $K \partial$, $I = 2$ curves passing through basepoint $z_1$ or $z_2$, and $I = 2$ buildings that pass through $p$. Now the $I=2$ buildings that pass through $p$ are exactly the buildings that have two $I=1$ levels, and the bottom level constitutes of trivial cylinders and $P_N$, by Lemma \ref{lem:curvethrup}. See Figure \ref{fig:Ubuilding}. Now since $\mathcal{M}_2^L$ is a compact $1$-dimensional manifold, we have that 
$$\partial \mathcal{M}_2^L = (\partial K)_{oo} + (K\partial)_{oo} + (U_{\#, z_1})_{oo} + (U_{\#, z_2})_{oo} + \frac{1}{h} \partial_{ho} = 0,$$
which is equivalent to (\ref{eq:U_homotopy}).
\end{remark}

\begin{figure}
    \centering
%    \includesvg{figs/Ubuilding}
    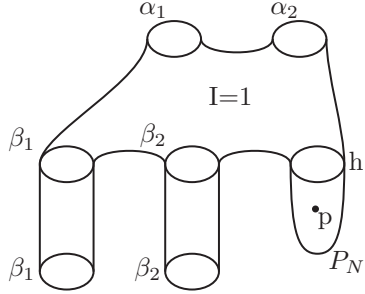
    \caption{An example of an $I = 2$ building that passes through $p$.}
    \label{fig:Ubuilding}
\end{figure}

\subsection{Proof of the main theorem}
\label{sec:proof_of_main}
Now we construct a chain homotopy equivalence $F: (C_{cone}^L, \partial_{cone})\to (C_\#^L, \partial_\#)$ in order to conclude the proof of Proposition \ref{prop:main_filtered}. 

\begin{lemma}
\label{lem:f_chain_homotopy_equivalence}
The map $F = \begin{pmatrix}
id & 0\\
hK_{oo} & h
\end{pmatrix}$ is a chain homotopy equivalence.
\end{lemma}

\begin{proof}
Since the diagonal constitutes of the identity maps on vector spaces and $F$ is lower triangular, it suffices to show that $F$ is a chain map, i.e. $F\partial_{cone} = \partial_\# F$.

This amounts to checking the following equation:
$$\begin{pmatrix}
id & 0\\
hK_{oo} & h
\end{pmatrix}
\begin{pmatrix}
\partial_{oo} & 0\\
U_1\otimes id + id \otimes U_2 & \partial_{oo}
\end{pmatrix}=
\begin{pmatrix}
\partial_{oo} & \partial_{oh}\\
\partial_{ho} & \partial_{hh}
\end{pmatrix}\begin{pmatrix}
id & 0\\
hK_{oo} & h
\end{pmatrix}
$$
Or, equivalently:
\begin{enumerate}
    \item $id \circ\partial_{oo} + 0 = \partial_{oo}\circ id +\partial_{oh}h K_{oo}$. This is true by Corollary \ref{cor:L_oh}.

    \item $0 + 0 = 0 + \partial_{oh} h$. This is true by Corollary \ref{cor:L_oh}. 
    
    \item $hK_{oo}\partial_{oo} + h(U_1\otimes id + id \otimes U_2) = \partial_{ho}\circ id + \partial_{hh}h K_{oo}$. This is true by Corollary \ref{cor:L_hh} and Lemma \ref{lem:umaphomotopy}.
    
    \item $0 + h\partial_{oo} = 0 + \partial_{hh}h$. This is true by Corollary \ref{cor:L_hh}.
\end{enumerate}
\end{proof}

\begin{proof}[Proof of Proposition \ref{prop:main_filtered}]
    The inverse of $F$ in Lemma \ref{lem:f_chain_homotopy_equivalence} gives the chain homotopy equivalence $f$.
\end{proof}

\section{The direct limit}
\label{sec:direct_limit}
We have shown the filtered version Proposition \ref{prop:main_filtered} of the main theorem. Now to obtain Theorem \ref{thm:main}, we use a direct limit argument similar to that in \cite{Nelson_Morgan} which involves Seiberg-Witten theory. First, we construct an isomorphism on the cobordism that interpolates between connected sums with necks of different radii up to a certain threshold. The following lemma is useful for constructing the cobordism.

\begin{lemma}
\label{lem:convex_comb_contact}
    Let $\alpha$ and $\alpha'$ be two contact forms of the same contact structure $\xi$ on a contact manifold $Y$, then the convex combination $\lambda:=(1-s)\alpha+s\alpha'$, where $s\in [0,1]$, is a contact form.
\end{lemma}
\begin{proof}
    We have that 
    $$d\lambda = -ds\wedge \alpha + (1-s)d\alpha + ds\wedge \alpha' + sd\alpha'$$ and 
    $$\lambda \wedge d\lambda = (1-s)^2\alpha \wedge d\alpha + \alpha \wedge ds \wedge \alpha' + (1-s)s \alpha \wedge d\alpha' + s(1-s) \alpha' \wedge d\alpha + s^2 \alpha' \wedge d\alpha'.$$
    Now since $\ker \alpha = \ker \alpha'$, we have that $\alpha = \varphi \alpha'$ for some positive function $\varphi: Y \to \R_{>0}$. Therefore, 
    $$\lambda \wedge d\lambda = (1-s)(1+\frac{s}{\varphi}) \alpha \wedge d\alpha + s(s+(1-s)\varphi)\alpha' \wedge d\alpha'>0$$ 
    since $\alpha$ and $\alpha'$ are contact forms.
\end{proof}

\begin{definition}
An \emph{admissible deformation} is a smooth family $\rho = \{(\lambda_t, L, J_t, r_t) \ |\ t\in [0,1]\}$ such that for each $t\in [0,1]$:
\begin{itemize}
    \item $\lambda_t$ is an $L_t$-nondegenerate contact form on $Y$;
    \item $J_t$ is a $\lambda_t$-adapted almost complex structure;
    \item $r_t$ is a positive real number.
\end{itemize}
\end{definition}

\begin{proposition}
\label{prop:cobordism}
Fix $L$ such that no orbit sets in $Y_1$ and $Y_2$ have action $L$. Then for 
$$r'< r \leq R(L),$$ 
there is a cobordism map
\begin{equation}
g^L_{r,r'}: ECH_*^{L}(Y_1\#_{r} Y_2 , \alpha_{r}) \rightarrow ECH_*^{L}(Y_1\#_{r'} Y_2 , \alpha_{r'}),
\end{equation}
which is an isomorphism.

\end{proposition}

\begin{proof}
Consider the admissible deformation 
\begin{equation}
\label{eq:admissible_deformation}
    \rho:=\{(\lambda_t:=(1-t)\alpha_{r'} + t\alpha_{r}, L, J_t, \mu)|t\in [0,1]\},
\end{equation}
where $\mu$ is sufficiently large. Note that $\lambda_t$ is a contact form since a convex combination of contact forms of the same contact structure is still a contact form by Lemma \ref{lem:convex_comb_contact}.

Now, observe by Lemma \ref{lem:large_energy} that for $r'<r\leq R(L)$, we have that any orbit intersecting the connected sum $S^2$ has to have action $\geq L$, except for the special hyperbolic orbit whose action stays $<<L$. Since the symplectic cobordism $([0,1] \times (Y_1 \# Y_2), d(e^t\lambda_t))$ does not change orbits away from the connected sum region by the proof of Lemma \ref{lem:one_handle}, no orbit set of action equal to $L$ appears for any $\lambda_t$.

By Lemma 3.7 and Section 3.5 in \cite{HT_Chord_II}, since $\lambda_t$ has no orbit sets of action $L$ and we may perturb $J$ to be $ECH^{L}$-generic, we have that for each $\Gamma\in H_1(Y)$, a well-defined isomorphism 
\begin{equation}
\label{eq:filtered_isom_ECH_HM}
    ECH_*^L(Y, \lambda_t, \Gamma) \cong \widehat{HM}_L^{-*}(Y,\lambda_t, \mathfrak{s}_{\xi, \Gamma}),
\end{equation}

where $\mathfrak{s}_{\xi, \Gamma}$ is the spin-c structure $\mathfrak{s}(\xi)+ \text{PD}(\Gamma)$.

For the definition of $\widehat{HM}_L^{-*}(Y,\lambda_t, \mathfrak{s}_{\xi, \Gamma})$, see \cite{HT_Chord_II}, whose exact definition we do not need in this proof.

By Lemma 3.4 in \cite{HT_Chord_II}, the admissible deformation $\rho$ gives an isomorphism
\begin{equation}
\label{eq:filtered_HM_admissible}
    \Phi: \widehat{HM}_L^{-*}(Y, \mathfrak{s}_{\xi, \Gamma}; \lambda_0, J_0, \mu) \stackrel{\simeq}{\longrightarrow} \widehat{HM}_L^{-*}(Y, \mathfrak{s}_{\xi, \Gamma}; \lambda_1, J_1, \mu).
\end{equation}
Composing the above isomorphisms (\ref{eq:filtered_isom_ECH_HM}) and (\ref{eq:filtered_HM_admissible}) gives the desired map $g^L_{r,r'}$. It is shown in Section 3.3 in \cite{HT_Chord_II} that $g^L_{r,r'}$ is well-defined.
\end{proof}

Given $L$, when $r>R(L)$, we no longer have the nice control of $ECC^L(Y_1 \#_r Y_2)$, since there might be orbits of large action (despite being $<L$) that cross the connected sum $S^2$. However, we still have the following cobordism map induced by an exact symplectic cobordism given in \cite{HT_Chord_II}.

\begin{theorem}[Corollary 5.3(a) in \cite{HT_Chord_II}]
\label{thm:exact_cobordism_SW}
    Let $(X,\lambda)$ be an exact symplectic cobordism from $(Y_+, \lambda_+)$ to $(Y_-, \lambda_-)$, where $\lambda_\pm$ is $L$-nondegenerate. Let $J_\pm$ be a symplectization-adapted almost complex structure for $\lambda_\pm$. Suppose $\rho$ is sufficiently large. Fix appropriate $2$-forms $\mu_\pm$ and perturbations needed to define the chain complex $\widehat{CM}^*(Y_\pm; \lambda_\pm, J_\pm, \rho)$. Then there is a well-defined map
    \begin{equation}
        {\widehat{HM}_L}^*(X, \lambda):{\widehat{HM}_L}^*(Y_+; \lambda_+, J_+, \rho)\longrightarrow {\widehat{HM}_L}^*(Y_-; \lambda_-, J_-, \rho)
    \end{equation}
    depending only on $X, \lambda, L, \rho, J_\pm, \mu_\pm$ and the perturbations, such that
\begin{equation}
\label{diag:SW_four}
\begin{tikzcd}
{\widehat{HM}_L}^*(Y_+; \lambda_+, J_+, \rho) \arrow{r}{i^{L, L'}} \arrow[swap]{d}{{\widehat{HM}_{L}}^*(X, \lambda)} & {\widehat{HM}_{L'}}^*(Y_+; \lambda_+, J_+, \rho) \arrow{d}{{\widehat{HM}_{L'}}^*(X, \lambda)} \\%
{\widehat{HM}_{L}}^*(Y_-; \lambda_-, J_-, \rho) \arrow{r}{i^{L,L'}}& {\widehat{HM}_{L'}}^*(Y_-; \lambda_-, J_-, \rho)
\end{tikzcd}
\end{equation}
commutes for $L<L'$, where $i^{L,L'}$ are induced by inclusions of chain complexes.
\end{theorem}

Therefore, we may consider the product exact symplectic cobordism 
$$(X := [0,1]_t \times (Y_1 \# Y_2), \lambda_t),$$
where $\lambda_t$ is defined in (\ref{eq:admissible_deformation}), from $Y_1 \#_{r'} Y_2$ to $Y_1 \#_r Y_2$ where $r'<r$. After passing through the isomorphism to Seiberg-Witten as in (\ref{eq:filtered_isom_ECH_HM}), Theorem \ref{thm:exact_cobordism_SW} gives the following cobordism map:
\begin{equation}
\label{eq:exact_cobord}
    ECH_*^L(X, \lambda): ECH_*^L(Y_1 \#_{r} Y_2, \alpha_r) \longrightarrow ECH_*^L(Y_1 \#_{r'} Y_2, \alpha_{r'}).
\end{equation}

Going back to the case when $r\leq R(L)$, the cobordism map induced by admissible deformation in Proposition \ref{prop:cobordism}, together with the cobordism map induced by inclusion, gives the following commutative diagram (as in proof of Lemma 3.7 in \cite{HT_Chord_II}) for $L_1<L_2$ and $R(L_1)>R(L_2)$:

\begin{equation}
\label{diag:four}
\begin{tikzcd}
ECH_*^{L_1}(Y_1\#_{R(L_1)} Y_2 , \alpha_{R(L_1)}) \arrow{r}{i^{L_1, L_2}} \arrow[swap]{d}{g_{R(L_1),R(L_2)}^{L_1}} & ECH_*^{L_2}(Y_1\#_{R(L_1)} Y_2 , \alpha_{R(L_1)}) \arrow{d}{g_{R(L_1),R(L_2)}^{L_2}} \\%
ECH_*^{L_1}(Y_1\#_{R(L_2)} Y_2 , \alpha_{R(L_2)}) \arrow{r}{i^{L_1, L_2}}& ECH_*^{L_2}(Y_1\#_{R(L_2)} Y_2, \alpha_{R(L_2)}),
\end{tikzcd}
\end{equation}
where $i^{L_1, L_2}$ are inclusion induced cobordism maps defined in Theorem \ref{thm:filtered_ech}. 
Furthermore, by Lemma 5.6 in \cite{HT_Chord_II}, the map $g_{R(L_1),R(L_2)}^{L_i}$ identifies with the exact cobordism map $ECH_*^L(X, \lambda)$ in (\ref{eq:exact_cobord}) when we consider $L = L_i$, $r = R(L_1)$ and $r' = R(L_2)$.

Now we check that (\ref{diag:four}) gives a directed system, similar to that in \cite{Nelson_Morgan}. We need to check that we have a well-defined composition for the cobordism maps
\begin{equation}
\label{eq:composition}
    \Phi^{L_1, L_2}(R(L_1),R(L_2)):ECH_*^{L_1}(Y_1\#_{R(L_1)} Y_2 , \alpha_{R(L_1)}) \to ECH_*^{L_2}(Y_1\#_{R(L_2)} Y_2 , \alpha_{R(L_2)})
\end{equation}
defined by either path in (\ref{diag:four}). For $L_1<L_2<L_3$, we define $r:=R(L_1)$, $r':=R(L_2)$ and $r'':=R(L_3)$. Then the composition of (\ref{eq:composition}) is given by the following four-fold commutative diagram.

\[ \begin{tikzcd}
ECH_*^{L_1}(Y_1\#_{r} Y_2 , \alpha_{r}) \arrow{r}{i^{L_1, L_2}} \arrow[swap]{d}{g_{r,r'}^{L_1}} 
& ECH_*^{L_2}(Y_1\#_{r} Y_2 , \alpha_{r}) \arrow{r}{i^{L_2, L_3}} \arrow{d}{g_{r,r'}^{L_2}} 
& ECH_*^{L_3}(Y_1\#_{r} Y_2 , \alpha_{r}) \arrow{d}{g_{r,r'}^{L_3}} \\%
ECH_*^{L_1}(Y_1\#_{r'} Y_2 , \alpha_{r'}) \arrow{r}{i^{L_1, L_2}}
\arrow[swap]{d}{g_{r',r''}^{L_1}} 
& ECH_*^{L_2}(Y_1\#_{r'} Y_2 , \alpha_{r'}) \arrow{r}{i^{L_2, L_3}} \arrow{d}{g_{r',r''}^{L_2}} 
& ECH_*^{L_3}(Y_1\#_{r'} Y_2 , \alpha_{r'}) \arrow{d}{g_{r',r''}^{L_3}} \\%
ECH_*^{L_1}(Y_1\#_{r''} Y_2 , \alpha_{r''}) \arrow{r}{i^{L_1, L_2}}
& ECH_*^{L_2}(Y_1\#_{r''} Y_2 , \alpha_{r''}) \arrow{r}{i^{L_2, L_3}} 
& ECH_*^{L_3}(Y_1\#_{r''} Y_2 , \alpha_{r''}) 
\end{tikzcd}
\]

Now, we may pass the filtered ECH complexes to the direct limit with respect to the above maps.  This requires some algebraic manipulations as in proof of Theorem 7.1 in \cite{Nelson_Morgan}, with which we will conclude the proof of our main theorem.

\begin{proof}[Proof of Theorem \ref{thm:main}]
In this proof we suppress the notation of the connected sum radius in $Y_1 \# Y_2$ when it is encoded in the contact form, in order to simplify notations. Let $L(r)$ denote the value of $L$ such that $R(L) = r$, where $R(L)$ is the strictly decreasing function defined in Section \ref{sec:filtered_ech_of_a_connected_sum}. Therefore, 
$$ECH^L(Y_1 \# Y_2, \alpha_{R(L)}) = ECH^{L(r)}(Y_1 \# Y_2, \alpha_r).$$ 
Now,
\begin{equation}
\label{eq:direct_limit_eq_chains}
\begin{split}
    ECH(Y_1\# Y_2, \xi_1 \# \xi_2) 
    &= \lim\limits_{r \to 0} ECH(Y_1 \# Y_2, \alpha_r) \\
    &= \lim\limits_{r \to 0}\lim\limits_{L \to \infty} ECH^L(Y_1 \# Y_2, \alpha_r) \\
    &= \lim\limits_{r \to 0}ECH^{L(r)}(Y_1 \# Y_2, \alpha_r) \\
    &= \lim\limits_{L \to \infty} ECH^L(Y_1 \# Y_2, \alpha_{R(L)}) \\
    &= H_*(Cone(U_1 \otimes id + id \otimes U_2)).
\end{split}
\end{equation}
The first equation is because ECH is independent of the choice of contact forms \cite{Taubes_isomorphism_I}. The second equation is given by (\ref{eq:direct_lim_J}) and Theorem \ref{thm:filtered_ech}(1). The third equation comes from consideration of the following map:
\begin{equation}
\label{eq:psi}
    \Psi: \lim\limits_{r \to 0}ECH^{L(r)}(Y_1 \# Y_2, \alpha_r) \longrightarrow \lim\limits_{r \to 0}\lim\limits_{L \to \infty} ECH^L(Y_1 \# Y_2, \alpha_r)
\end{equation}
by sending the equivalence class of an element $d_r\in ECH^{L(r)}(Y_1 \# Y_2, \alpha_r)$ under $\lim\limits_{r \to 0}$ to the equivalence class of $d_r$ under $\lim\limits_{r \to 0}\lim\limits_{L \to \infty}$.
We need to show that $\Psi$ is well-defined and a bijection. To establish well-definedness, consider $d_{r}$ and $d_{r'}$ in $\lim\limits_{r \to 0}ECH^{L(r)}(Y_1 \# Y_2, \alpha_r)$ where $r>r'$ and $d_{r}\sim d_{r'}$. This means that there is a common element $d_{r''}\in ECH^{L(r'')}(Y_1 \# Y_2, \alpha_{r''})$ such that both $d_{r}$ and $d_{r'}$ are mapped to under the direct limit. Then $\Psi(d_{r}) \sim \Psi(d_{r'})$ by composing the maps
\begin{equation}
    ECH^{L(r)}(Y_1 \# Y_2, \alpha_r) \longrightarrow ECH^{L(r'')}(Y_1 \# Y_2, \alpha_r) \longrightarrow ECH^{L(r'')}(Y_1 \# Y_2, \alpha_{r''})
\end{equation}
and
\begin{equation}
    ECH^{L(r')}(Y_1 \# Y_2, \alpha_{r'}) \longrightarrow ECH^{L(r'')}(Y_1 \# Y_2, \alpha_{r'}) \longrightarrow ECH^{L(r'')}(Y_1 \# Y_2, \alpha_{r''})
\end{equation}
given by the right followed by down composition of the commutative diagram (\ref{diag:four}). Now we show that $\Psi$ is injective. Suppose $\Psi(d) = 0$. That means there exists $L_0$ such that a representative $\widetilde{\Psi(d)} \in ECH^{L_0}(Y_1 \# Y_2, \alpha_{r_0})$, where $[\widetilde{\Psi(d)}]=\Psi(d)$, is zero for some $r_0$. See Figure \ref{fig:direct_limit}. If $r_0\leq R(L_0)$, then we are done, since for $r\leq R(L)$, 
$$ECH^{L}(Y_1 \# Y_2, \alpha_r) =  ECH^{L}(Y_1 \# Y_2, \alpha_{R(L))})$$ by Proposition \ref{prop:cobordism}, so (\ref{eq:psi}) is a bijection. Suppose $r_0 > R(L_0)$. Then $\widetilde{\Psi(d)}$ is mapped to zero in $ECH^{L_0}(Y_1 \# Y_2, \alpha_{R(L_0)})$ under the map in (\ref{eq:exact_cobord}). Therefore, $d\sim 0$. To show that $\Psi$ is surjective, let 
$$y\in \lim\limits_{L \to \infty} ECH^L(Y_1 \# Y_2, \alpha_{r_0})$$ 
for some $r_0$. Similar to the injective case, if $r_0\leq R(L_0)$, then we are done. Suppose $r_0 > R(L_0)$. Now let $x\in ECH^{L_0}(Y_1 \# Y_2, \alpha_{r_0})$ be such that $[x]=y$. Let 
$$x'\in ECH^{L_0}(Y_1 \# Y_2, \alpha_{R(L)})$$ 
be the image of $x$ when taking the limit as $r\to 0$ defined by the exact cobordism map as in (\ref{eq:exact_cobord}). Then $[y] = \Psi([x'])$.

\begin{figure}
    \centering
%    \includesvg{figs/direct_limit}
    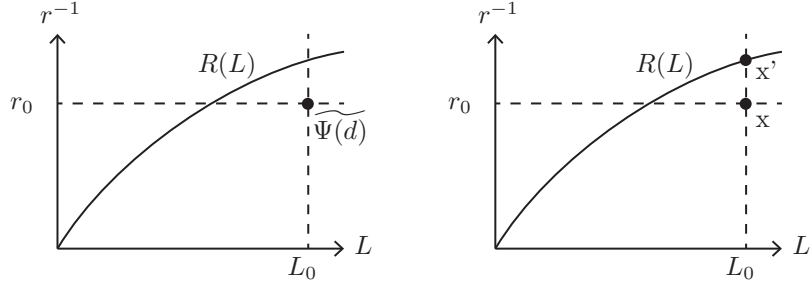
    \caption{Schematic illustration of the proof involving direct limits in Theorem \ref{thm:main}. In the region above $R(L)$, i.e. when $r\leq R(L)$, we have isomorphisms of $ECH_*^{L}(Y_1\#_{r} Y_2, \alpha_{r})$ given a fixed value of $L$ by Proposition \ref{prop:cobordism}.}
    \label{fig:direct_limit}
\end{figure}

The fourth equation is by the definition of $L(r)$. The fifth equation is by Proposition \ref{prop:main_filtered} and the fact that taking direct limit commutes with taking homology. 

Finally, note that the mapping cone in (\ref{eq:direct_limit_eq_chains}) is of the map 
$$U_1 \otimes id + id \otimes U_2 : ECC(Y_1, \lambda_1) \otimes_{\F} ECC(Y_2, \lambda_2) \longrightarrow  ECC(Y_1, \lambda_1) \otimes_{\F} ECC(Y_2, \lambda_2)[-1].$$ 
Over $\F$, the homology of $Cone(U_1 \otimes id + id \otimes U_2)$ is isomorphic to the homology of the mapping cone of the induced map $(U_1 \otimes id + id \otimes U_2)_*$ on homology, by observing that K\"{u}nneth formula is natural over $\F$. This concludes the proof of the main theorem.
\end{proof}

\printbibliography
\end{document}